%% file: centralisingmonoids4maj.tex
\RequirePackage{fix-cm}%
\documentclass[paper=a4,fontsize=12pt,oneside,cleardoublepage=empty,listof=totoc,bibliography=totoc,index=totoc,BCOR13.0mm]{scrartcl}%
\usepackage[final]{graphicx}
\usepackage[utf8]{inputenc}
\usepackage[T2A,T1]{fontenc}
\usepackage{lmodern,fixltx2e}%
\usepackage{etex}%
\usepackage[russian,german,british]{babel}
\usepackage{hyphxcpt}
\usepackage{bm}
\usepackage{tikz}
\newif\ifforceblacklinks\forceblacklinksfalse

\def\PublicationTitle{All centralising monoids with majority witnesses on a four-element set}%
\def\CorrespondingAuthor{Mike Behrisch}%

\def\TUWname{\foreignlanguage{german}{%
             Tech\-ni\-sche Uni\-ver\-si\-t\"{a}t Wien}}%
\def\InstitutDMG{\foreignlanguage{german}{%
             In\-sti\-tut f\"{u}r Dis\-kre\-te Ma\-the\-ma\-tik und
             Geo\-me\-trie}}%
\def\AddressDMG{\foreignlanguage{german}{%
             Wied\-ner Haupt\-str.~8--10/E104}}%
\def\PostleitzahlWien{\mbox{A-1040} Vienna, Austria}%
\def\PublicationTopic{monoids defined as centralisers of
                      majority operations or semiprojections
                      on the four-element set}%
\def\PublicationKeywords{commutation,
                         centraliser clone,
                         bicentrically closed clone,
                         primitive positive clone,
                         centralising monoid,
                         majority operation,
                         semiprojection%
                         }

\usepackage{amsmath}
\usepackage{amssymb}
\usepackage{amsfonts}
\usepackage{enumerate}
\usepackage[mathscr]{eucal}
\usepackage{url}

\usepackage{mathoperators}
\usepackage[uglyemptyset,hyperref]{usefulthings}
\usepackage{PolInv}
\usepackage[plain,proof,section]{thmdefinitions}

\makeatletter
\def\DefaultCarrier{\PolInv@defaultsetA}
\newcommand{\CarrierSet}[1][]{
  \ifthenelse{\equal{#1}{}}{%
    \ensuremath{\DefaultCarrier}
  }{%
    \ensuremath{\DefaultCarrier^{#1}}
  }
}
\DeclareMathOperator{\ConstantsOp}{C}%
\newcommand{\Constants}[2][]{\@OpWithOptionalArityAndCombination{#1}{#2}{\ConstantsOp}{\@combinewithindex}}
\DeclareRobustCommand{\Consts}[1][]{\Constants[#1]{\DefaultCarrier}}%
\newcommand{\Ops}[1][]{\Op[#1]{\DefaultCarrier}}%
\newcommand{\TrivOps}[1][]{\J[#1]{\DefaultCarrier}}%
\newcommand*{\Palatalization}[1]{%
  \bgroup\fontencoding{T1}\selectfont\v{#1}\egroup}
\theoremdefinitions@styleForThm
\newtheorem{observation}[thmdefinitions@Regelsatz]{Observation}
\makeatother

\def\graph{}%
\renewcommand{\graph}[1]{\ensuremath{#1^{\bullet}}}%
\newcommand{\cent}[1]{\ensuremath{{#1^{*}}}}%
\newcommand{\bicent}[1]{\cent{\cent{#1}}}%
\newcommand{\commuteswith}{\mathrel{\bot}}%
\newcommand{\crd}[1]{\abs{#1}}%
\DeclareMathOperator{\MajOp}{Maj}%
\newcommand{\Maj}[1][\CarrierSet]{\ensuremath{\MajOp_{#1}}}%
\DeclareMathOperator{\bigO}{O}%
\newcommand{\LandauO}[1]{\bigO\apply{#1}}%
\newcommand{\cb}{,\mathclose{}\mathbin{\rule{0pt}{0pt}}}%
\newcommand{\mathff}[1]{{\bm#1}}%
\newcommand{\Restr}[2]{#1\rvert_{#2}}%
\newcommand{\condcyclic}[5]{%
\m{f\commuteswith \apply{#1 #2 #3 #4}} if and only if each of the following
quadruples belongs to the set
\m{\set{\apply{#1,#2,#3,#4},\apply{#2,#3,#4,#1},\apply{#3,#4,#1,#2},\apply{#4,#1,#2,#3}}}:
\bgroup\footnotesize
\begin{align*}
\apply{f\apply{#1,#2,#3},f\apply{#2,#3,#4},f\apply{#3,#4,#1},f\apply{#4,#1,#2}}, \\
\apply{f\apply{#1,#2,#4},f\apply{#2,#3,#1},f\apply{#3,#4,#2},f\apply{#4,#1,#3}}, \\
\apply{f\apply{#1,#3,#2},f\apply{#2,#4,#3},f\apply{#3,#1,#4},f\apply{#4,#2,#1}}, \\
\apply{f\apply{#1,#4,#2},f\apply{#2,#1,#3},f\apply{#3,#2,#4},f\apply{#4,#3,#1}}, \\
\apply{f\apply{#1,#3,#4},f\apply{#2,#4,#1},f\apply{#3,#1,#2},f\apply{#4,#2,#3}}, \\
\apply{f\apply{#1,#4,#3},f\apply{#2,#1,#4},f\apply{#3,#2,#1},f\apply{#4,#3,#2}}.\tag{#5}
\end{align*}
\egroup
}
\newcommand{\condtwocycles}[5]{%
\m{f\commuteswith\apply{#1 #2}\apply{#3 #4}} if and only if each of the
following sets either equals \m{\set{#1,#2}} or \m{\set{#3,#4}}:
\bgroup\footnotesize
\begin{align*}
  \set{f\apply{#1,#2,#3},f\apply{#2,#1,#4}}, &&
  \set{f\apply{#1,#3,#4},f\apply{#2,#4,#3}}, &&
  \set{f\apply{#3,#1,#2},f\apply{#4,#2,#1}}, \\
  \set{f\apply{#1,#2,#4},f\apply{#2,#1,#3}}, &&
  \set{f\apply{#1,#4,#3},f\apply{#2,#3,#4}}, &&
  \set{f\apply{#3,#4,#1},f\apply{#4,#3,#2}}, \\
  \set{f\apply{#1,#3,#2},f\apply{#2,#4,#1}}, &&
  \set{f\apply{#3,#1,#4},f\apply{#4,#2,#3}}, &&
  \set{f\apply{#4,#1,#2},f\apply{#3,#2,#1}}, \\
  \set{f\apply{#1,#4,#2},f\apply{#2,#3,#1}}, &&
  \set{f\apply{#4,#1,#3},f\apply{#3,#2,#4}}, &&
  \set{f\apply{#4,#3,#1},f\apply{#3,#4,#2}}.
  \tag{#5}
\end{align*}
\egroup
}
\newcommand{\condtriple}[5]{%
\m{f\commuteswith\apply{#1 #2 #3}\apply{#4}} if and only if each of the
following triples belongs to the set
\m{\set{\apply{#1,#2,#3},\apply{#2,#3,#1},\apply{#3,#1,#2},\apply{#4,#4,#4}}}:
\bgroup\footnotesize
\begin{align*}
  \apply{f\apply{#1,#2,#3},f\apply{#2,#3,#1},f\apply{#3,#1,#2}}, &&
  \apply{f\apply{#1,#3,#2},f\apply{#2,#1,#3},f\apply{#3,#2,#1}}, \\
  \apply{f\apply{#1,#2,#4},f\apply{#2,#3,#4},f\apply{#3,#1,#4}}, &&
  \apply{f\apply{#2,#1,#4},f\apply{#3,#2,#4},f\apply{#1,#3,#4}}, \\
  \apply{f\apply{#1,#4,#2},f\apply{#2,#4,#3},f\apply{#3,#4,#1}}, &&
  \apply{f\apply{#2,#4,#1},f\apply{#3,#4,#2},f\apply{#1,#4,#3}}, \\
  \apply{f\apply{#4,#1,#2},f\apply{#4,#2,#3},f\apply{#4,#3,#1}}, &&
  \apply{f\apply{#4,#2,#1},f\apply{#4,#3,#2},f\apply{#4,#1,#3}}.
  \tag{#5}
\end{align*}%
\egroup%
}
\newcommand{\condtwofixedpoints}[5]{%
\m{f\commuteswith\apply{#1 #2}\apply{#3}\apply{#4}}
if and only if each of the following sets either equals
\m{\set{#1,#2}}, \m{\set{#3}} or \m{\set{#4}}:
\bgroup\footnotesize
\begin{align*}
  \set{f\apply{#1,#2,#3},f\apply{#2,#1,#3}}, &&
  \set{f\apply{#3,#1,#2},f\apply{#3,#2,#1}}, &&
  \set{f\apply{#3,#1,#4},f\apply{#3,#2,#4}}, \\
  \set{f\apply{#1,#2,#4},f\apply{#2,#1,#4}}, &&
  \set{f\apply{#4,#1,#2},f\apply{#4,#2,#1}}, &&
  \set{f\apply{#4,#1,#3},f\apply{#4,#2,#3}}, \\
  \set{f\apply{#1,#3,#2},f\apply{#2,#3,#1}}, &&
  \set{f\apply{#1,#3,#4},f\apply{#2,#3,#4}}, &&
  \set{f\apply{#3,#4,#1},f\apply{#3,#4,#2}}, \\
  \set{f\apply{#1,#4,#2},f\apply{#2,#4,#1}}, &&
  \set{f\apply{#1,#4,#3},f\apply{#2,#4,#3}}, &&
  \set{f\apply{#4,#3,#1},f\apply{#4,#3,#2}}.
  \tag{#5}
\end{align*}%
\egroup%
}

\hypersetup{
    pdftitle={\PublicationTitle},    
    pdfauthor={\CorrespondingAuthor},     
    pdfsubject={\PublicationTopic},   
    pdfkeywords={\PublicationKeywords},         
}
\ifforceblacklinks%
\hypersetup{%
    linkcolor=black,     
    citecolor=black,     
    filecolor=black,     
    urlcolor=black       
}%
\fi%
\numberwithin{equation}{section} 

\begin{document}

\thispagestyle{empty}
\selectlanguage{british}

\title{\PublicationTitle%
       \thanks{This research was partially supported by the BMBWF
               through the OeAD project KONTAKT CZ~02/2019 `Function
               algebras and ordered structures related to logic and
               data fusion'.}}
\author{\CorrespondingAuthor%
        \thanks{\TUWname, \InstitutDMG}}%
\date{\today}
\maketitle

\begin{abstract}
We study unary parts of centraliser clones on the set \m{\set{0,1,2,3}},
so\dash{}called centralising monoids. We describe and count all
centralising monoids on the set \m{\set{0,1,2,3}} having
majority operations as
witnesses, and we list the inclusion maximal proper submonoids of the
full transformation monoid among them.
This extends previous work by Goldstern, Machida and Rosenberg.
\bgroup
\let\thefootnote\relax%
\footnote{%
\par%
}%
\footnote{This is a preliminary and partially extended version of an
article to appear in \emph{J. Mult.-Valued Logic Soft Comput.}}
\setcounter{footnote}{0}%
\egroup
\end{abstract}

\section{Introduction}\label{sect:intro}
Besides for being a very gentle and warm\dash{}hearted person, a gifted
but, at the same time, humble and unassuming mathematician, Ivo
Rosenberg was---and still is---probably best known in the universal
algebraic community for his two outstanding contributions to clone
theory: the characterisation of all maximal clones on finite sets
(\cite{RosenbergMaximalClonesFrench,RosenbergMaximalClones}) and his
classification theorem for generating functions of minimal clones
(\cite{RosenbergMinimalClones}), which will also feature in this
article (Theorem~\ref{thm:Rosenberg-minimal-clones}). However, his work on clones did not stop with clones on finite
sets, he also contributed to the theory of clones on infinite sets
(notably~\cite{RosenbergNumberOfClonesOnInfiniteSet,RosenbergNumberOfMaximalClonesInfiniteSet}),
and he studied generalisations of clones, like partial clones
(\cite{HaddadRosenbergPartialClonesContainingSymA,HaddadMachidaRosenbergMaximalMinimalPartialClones,HaddadLauRosenbergIntervalsOfPartialClonesOverMaximalClones,CouceiroHaddadRosenbergPartialClonesMonotoneSelfDualOps}),
hyperclones
(\cite{RosenbergHyperclones,MachidaPantovicRosenbergRegularSetsOfOperationsHyperclones})
and multiclones (e.g.~\cite{PouzetRosenbergMulticlones}), as well.
In this article we will directly build upon some of his research.
\par
It is well understood that every clone on a finite set can be written
as the set of polymorphisms of some set of relations
(\cite{GeigerClosedSystemsOfFunctionsAndPredicates,BodnarcukKaluzninKotovRomovGaloisTheoryForPostAlgebras}),
and this Galois theoretic approach to clones is central to Rosenberg's
work on clones and related structures. If one restricts in this
Galois connection the role of relations to allow only graphs of
functions, one obtains the notion of \emph{commutation} between
functions: In this context we say that an \nbdd{n}place function~$f$
commutes with an \nbdd{m}place function~$g$, both defined over the same
set of arguments~$\CarrierSet$, if
\m{g\apply{f\apply{X}} = f\apply{g\apply{X^T}}} holds
for any matrix \m{X\in \CarrierSet[m\times n]}, where
\m{f(X)\defeq\apply{f\apply{X\apply{i,\cdot}}}_{i\in m}}
denotes the tuple obtained by applying~$f$
row\dash{}wise to the matrix~$X$, and similarly for~$g$ and the
transposed matrix~$X^T$ (i.e., $g$ operates on the columns of~$X$).
This commutation condition, denoted by $f \commuteswith g$, is the same
as saying that~$f$ preserves the graph \m{\graph{g}\defeq
\lset{(x,g(x))}{x\in\CarrierSet[m]}\subs A^{m+1}}, understood as an
\nbdd{(m+1)}ary relation on~$\CarrierSet$; moreover, it is easy to see
that it is a symmetric property, i.e., \m{f\commuteswith g} if and only
if \m{g\commuteswith f}.
Now collecting, for a given set~$F$ of functions of finite (possibly
different) arity, the set of all finitary operations \m{g} that commute
with all functions from~$F$ gives the \emph{centraliser}
\m{\cent{F}=\lset{g}{\arity(g)<\omega\land \forall f\in F\colon
f\commuteswith g}}, which is the polymorphism clone
\m{\Pol{\graph{F}}} of all the graphs of the functions in~$F$
(where \m{\graph{F}\defeq \lset{\graph{f}}{f\in F}}).

Hence, the Galois closures of the Galois correspondence induced by
commutation are still clones (they are called \emph{centraliser
clones}), but they are far from being all of
them. In fact, there is only a finite number of centraliser clones for
every finite set~$\CarrierSet$, as was shown by Burris and
Willard~\cite[Corollary~4, p.~429]{BurrisWillardFinitelyManyPPClones},
while there are infinitely many (polymorphism) clones when
\m{\crd{\CarrierSet}\geq 2}.
By their definition, centraliser clones~$\cent{F}$ are exactly the
clones consisting of homomorphisms of finite powers of an algebra, namely
\m{\algwops{A}{F}}, to itself.
Alternatively, they can be characterised as those clones~$F$ that are
equal to their own \emph{bicentraliser}~$\bicent{F}$, i.e., those which
are \emph{bicentrically closed}: \m{F=\bicent{F}}.
\par
Despite (or perhaps because) there are so comparatively few of
them, there is not as much known about centraliser clones as about
clones in general.  Originally, they were introduced and studied
under the name \emph{parametrically closed classes} in work by
Kuznecov~\cite{KuznecovCentralisers1979} and
Dani\v{l}\v{c}enko~\cite{Danilcenko1977ParametricExpressibility3ValuedLogic,Danilcenko1977ParametricExpressibility3ValuedLogicTranslated,Danilcenko1978ParametricallyClosedClasses3ValuedLogic,Danilcenko1979-thesis,Danilcenko1979ParametricalExpressibilitykValuedLogic,KuznecovDanilcenkoNegruPavlovaCombinatorialQuestionskValuedLogic,Danilcenko1987PropertiesLatticeOfParametricallyClosedClasses3ValuedLogic}; later also
Harnau~\cite{Harnau1DefVertauschbarkeitMaximalitaet,Harnau2VertauschbarkeitBasisproblem,Harnau3NumInvarianten,Harnau4KettenlaengenVertauschbarkeitsmengen,Harnau5PosetVertauschbarkeitsmengen,Harnau6VerallgVertauschbarkeit}
added to the theory, mainly dealing with centralisers of unary
operations. Continuing the line of studies initiated by Harnau,
centraliser clones also form part of the vast \oe{}uvre of Ivo
Rosenberg,
e.g.~\cite{MachidaMiyakawaRosenbergRelationsBetwClonesAndFullMonoids,MachidaMiyakawaRosenbergCentralisersOfMonoidsInCloneTheory,MachidaRosenbergCentralisersOfMonoidsInCloneTheory,MachidaRosenbergCentralizersClonesOfMonoidsBeingTrivial,MachidaRosenbergCentralizersClonesOfMonoidsContainingSymA,MachidaRosenbergCentraliserClonesOfBinaryOps},
in particular, they appear through his investigations of
so\dash{}called \emph{centralising monoids}
\cite{MachidaRosenbergCentralisingMonoidsInCloneTheory,MachidaRosenbergCentralisingMonoidsWitnessLemma,MachidaRosenbergCentralisingMonoids,MachidaRosenberg3MaximalCentralizingMonoids3Elements,MachidaRosenbergGoldsternCentralisingMonoids,MachidaRosenbergCentralisingMonoidsBinaryIdempotentOps,MachidaRosenbergCentralisingMonoidsArity},
all jointly written with Hajime Machida.
This body of work is actually quite related to questions studied by
Harnau in the 1970ies, but the different authors approach the same
problem from `opposite' sides. Namely, they further restrict the Galois
connection of commutation to only allow unary operations on one of the
two sides (but functions of arbitrary arity on the other). Harnau
studies centraliser clones \m{\cent{\set{f}}} of unary functions
\m{f\colon \CarrierSet \to \CarrierSet} (and hence provides information
on centralisers \m{\cent{F} = \bigcap_{f\in F}\cent{\set{f}}} of
sets~$F$ of unary operations). Machida and Rosenberg also study the
other side of the Galois correspondence, namely the unary parts
\m{\Fn[1]{\cent{F}{}} = \lset{g\colon
\CarrierSet\to\CarrierSet}{g\in\cent{F}}} of centraliser
clones~\m{\cent{F}} given by arbitrary sets of functions~$F$. Since
these unary parts form transformation monoids on~$\CarrierSet$, they
are called \emph{centralising monoids}, and any set~$F$ of operations
confirming that some transformation monoid \m{M\subs
\CarrierSet[\CarrierSet]} is a centralising monoid
\m{M=\Fn[1]{\cent{F}{}}} is called a \emph{witness} of~$M$, see,
e.g.~\cite{MachidaRosenbergCentralisingMonoidsWitnessLemma,MachidaRosenbergCentralisingMonoids}.
Using Rosenberg's Classification Theorem for minimal clones
(Theorem~\ref{thm:Rosenberg-minimal-clones}) and an
exhaustive search through 84 cases,
in~\cite[Proposition~4.3, p.~156]{MachidaRosenbergCentralisingMonoids}
all ten maximal centralising monoids on a three\dash{}element set were
determined. In fact, the authors report
in~\cite[p.~205]{MachidaRosenbergGoldsternCentralisingMonoids} that they
enumerated all 192 centralising monoids on \m{\set{0,1,2}}, as well,
but for reasons of space they were not able to publish this long list
in~\cite{MachidaRosenbergCentralisingMonoids}. This exhaustive method
being somewhat unsatisfactory,
in~\cite{MachidaRosenbergGoldsternCentralisingMonoids}, the authors set
out for a more systematic approach to find all centralising monoids
on~\m{\set{0,1,2}} with majority functions or ternary semiprojections
as witnesses.
\par
The present paper is part of an endeavour to continue and properly
extend the work begun by Rosenberg and Machida
in~\cite{MachidaRosenbergCentralisingMonoids},
namely to obtain all maximal centralising monoids on a
\emph{four\dash{}element} domain. So, for a four\dash{}element
carrier set, e.g., \m{\CarrierSet=\set{0,1,2,3}}, we shall study the
Galois connection induced by commutation between unary operations and
functions of arbitrary positive arity on the other side, with the aim of
finding the co\dash{}atoms among the Galois closed sets on the side of
unary operations. For this we shall heavily rely on the methods
established in~\cite{MachidaRosenbergGoldsternCentralisingMonoids},
however, what could be done with pen and paper
in~\cite{MachidaRosenbergGoldsternCentralisingMonoids}, now requires
massive computerised support due to combinatorial explosion. Even
though the techniques developed
in~\cite{MachidaRosenbergGoldsternCentralisingMonoids} are systematic
in nature and our way of proceeding needs to refine these still more to
be at all feasible (cf.\ Section~\ref{sect:method}), the whole project
partly turns into a brute\dash{}force exhaustive search. However, such
is the nature of the beast, and thus a blow\dash{}up probably cannot be
avoided.
\par
Continuing the work done
in~\cite{MachidaRosenbergGoldsternCentralisingMonoids}, in this article
we investigate, as a first step, commutation between unary operations
and majority operations (and partly with ternary semiprojections). For
\m{\crd{\CarrierSet}=4} we show in
Corollary~\ref{cor:number-of-centralising-monoids} that there are
\emph{exactly $1\,715$ centralising monoids} having majority functions as
witnesses, and that there are~$147$ maximal ones among them
(Corollary~\ref{cor:maximal-centralising-monoids}). We provide enough
data so that anyone may check that there are at least $1\,715$
centralising monoids of this type. For this bound to be sharp, one has
to trust the completeness of our computations, the basis of which is
described in Section~\ref{sect:attribute-clarification} and the
beginning of Section~\ref{sect:results}. Since we are concerned with
the enumeration of Galois closures, we find it convenient to phrase our
line of action in the language of formal concept analysis (see,
e.g.~\cite{GaWiFCA}), as it offers pre\dash{}defined terminology for
the type of manipulations we are going to employ. A methodological part
explaining this connection can be found in Section~\ref{sect:method}.
Section~\ref{sect:prelims} introduces our notation and presents more
background information on maximal centralising monoids, which is needed
to understand the approach described in this paper and the relationship
between the presented results and the overall goal of finding all
maximal centralising monoids on \m{\set{0,1,2,3}}.
\par
It is worth noting that apart from its finiteness, practically nothing
is known about the lattice of centraliser clones beyond
three\dash{}element carrier sets. For the case
\m{\CarrierSet = \set{0,1,2}}, Dani\v{l}\v{c}enko gave a complete
description of the whole lattice in her doctoral
thesis~\cite[Section~6, p.~125 et seqq.]{Danilcenko1979-thesis},
showing that there are $2\,986$ centraliser clones. Already on
four\dash{}element domains, even just estimates for the size of the
lattice are unknown, and so any (lower) bounds, obtained, for instance,
by counting certain types of centralising monoids, provide valuable new
pieces of information.

\section{Preliminaries and notation}\label{sect:prelims}
\subsection{Fundamental definitions and notation}\label{subsect:notation}
In this article we shall use $\N=\omega=\set{0,1,2,\dots}$ to denote
the set of natural numbers including zero, i.e., all finite ordinals.
It will be convenient for us to employ the von Neumann model of
ordinals, so each natural number \m{n\in \N} will be the set
\m{n=\set{0,\dotsc,n-1}} of its (finite ordinal) predecessors, zero
being the empty set. In particular, we will often work with the finite
carrier set \m{\CarrierSet = \set{0,1,2,3} = 4}, and sometimes we shall
write just~$4$ for it (in other cases we will explicitly
use \m{A=\set{0,1,2,3}} to avoid confusion). Moreover, the cardinality
of some set~$B$ is denoted by~$\crd{B}$.
\par

A central object of this paper are functions, such as
\m{f\colon A\to B} and \m{g\colon B\to C}.
Their composition \m{g\circ f\colon A\to C} sends elements \m{x\in A} to \m{g(f(x))}, i.e., we
compose from the left. The set of all functions from~$A$ to~$B$ is
denoted by~$B^A$. If \m{f\in B^A} and \m{U\subs A}, \m{V\subs B}, the
image of~$U$ under~$f$ is \m{f\fapply{U} \defeq \lset{f(x)}{x\in U}},
and the preimage of~$V$ under~$f$ is
\m{f^{-1}\fapply{V} \defeq \lset{x\in A}{f(x)\in V}}. The image of~$f$
is just \m{\im f \defeq f\fapply{A}}. We also need restrictions; namely,
for \m{U\subs A} we have the function \m{\Restr{f}{U}\colon U\to B}
that operates as~$f$, but on a smaller domain (we leave the
co\dash{}domain unchanged). Besides, if \m{f\fapply{U}\subs U}, then we
can also modify the co\dash{}domain, to get \m{f\restriction_{U}},
which is a map from~$U$ to~$U$, sending \m{x\in U} to~\m{f(x)}.
\par

In particular, for \m{n\in\N} we understand \nbdd{n}tuples
\m{\mathff{x}\in\CarrierSet[n]} as functions from the set
\m{n=\set{0,\dotsc,n-1}} to \m{A}:
\m{\mathff{x}} is a map with values
\m{x_0=\mathff{x}(0), x_1=\mathff{x}(1),\dots} and is written as
\m{\mathff{x}=\apply{x_0,\dotsc,x_{n-1}}= \apply{x_i}_{i\in n}}. All
notation that is meaningful for functions will hence also be used for
tuples; this includes image, preimage and composition:
\m{\im{\mathff{x}} = \set{x_0,\dotsc,x_{n-1}}},
\m{\mathff{x}^{-1}\fapply{V} = \lset{0\leq i<n}{x_i\in V}} for
\m{V\subs\CarrierSet}, and
\m{f\circ\mathff{x} = \apply{f(x_0),\dotsc,f(x_{n-1})}}
for \m{f\colon \CarrierSet\to B}, while
\m{\mathff{x}\circ\alpha=\apply{x_{\alpha(i)}}_{i\in I}} for a
re\dash{}indexing function \m{\alpha\colon I\to n}.
\par

The functions we shall mainly be interested in have the form that their
domain is a finite power of their co\dash{}domain, i.e.,
\m{f\colon \CarrierSet[n]\to\CarrierSet} with \m{n\in\N} being the
\emph{arity} of~$f$. Functions of arity zero are negligible with
respect to commutation; hence we collect all \emph{finitary operations
on~$\CarrierSet$} in the set \m{\Ops = \bigcup_{n\in\N\setminus\set{0}}
A^{A^n}}. For any \m{F\subs\Ops} and \m{n\in\N} we denote by
\m{\Fn{F}\defeq A^{A^n}\cap F} its \nbdd{n}ary part. Hence, in
particular \m{\Ops[n]=A^{A^n}} and \m{\Fn{F}=\Ops[n]\cap F} for any
\m{F\subs\Ops}. For \m{i\in n\in\N} the \emph{\nbdd{n}ary projection}
onto the position~$i$ is the operation \m{\eni{i}\in\Ops[n]}
mapping \m{\mathff{x}=\apply{x_j}_{j\in n}} to
\m{\eni{i}(\mathff{x})\defeq x_i}. All projections form the set
\m{\TrivOps=\bigcup_{0<n<\omega}\lset{\eni{i}}{i\in n}}.
\par

Since centralisers are clones, we briefly mention that a \emph{clone}
is a set \m{F\subs\Ops} of operations, which contains all projections
(i.e., \m{\TrivOps\subs F}) and is closed under composition of finitary
operations. The precise meaning of `composition' in this context and
much more additional information on clones, e.g., on their Galois
theory, can be found in the reference monographs~\cite{PoeKal}
or~\cite{LauFunctionAlgebrasOnFiniteSets}. Easy examples of clones are
\m{\TrivOps,\Ops} and~\m{\cent{F}} for any \m{F\subs\Ops}.

To define interesting classes of finitary operations, it is convenient
to use the concept of \emph{identities}. An identity is an expression
of the form \m{s\approx t} where \m{s,t} are functional terms over a
given collection of function symbols and a certain supply of variable
symbols. \emph{Satisfaction} of an identity over a structure
$\algwops{\CarrierSet}{f,g,\dots}$, where \m{f,g,\dots} are concrete
functions interpreting the abstract symbols in the identity, means that
the equality described by $s\approx t$ is \emph{universally true},
i.e., is true when all variables in~$s$ or~$t$ are universally
quantified over~$\CarrierSet$ and the abstract symbols are replaced by
the concrete functions $f,g,\dots$ of the structure. When the carrier
set~$\CarrierSet$ is clear from the context, one usually suppresses the
structure and speaks of \emph{functions satisfying some identities};
moreover often the same symbols are used for the concrete functions and
the abstract symbols.
\par

For example, a projection on the place~$i$ is a function
\m{f\in\Ops[n]} satisfying the identity
\m{f(x_0,\dotsc,x_{n-1})\approx x_i}. An \emph{idempotent} function is
a finitary operation \m{f\in\Ops} satisfying \m{f(x,\dotsc,x)\approx x}.
A \emph{majority operation} is a ternary operation \m{f\in\Ops[3]}
satisfying the identities
\m{f(x,x,y)\approx f(x,y,x)\approx f(y,x,x)\approx x},
while a \emph{minority} operation \m{f\in\Ops[3]} satisfies
\m{f(x,x,y)\approx f(x,y,x)\approx f(y,x,x)\approx y}. Both majority
and minority operations are idempotent. A \emph{constant} is a function
\m{f\in\Ops[n]} satisfying \m{f(x_0,\dotsc,x_{n-1})\approx a} where
\m{a} is a constant (nullary) symbol from~$\CarrierSet$. We shall only
need unary constant operations, and denote them as \m{c_a} where
\m{a\in \CarrierSet} is the unique value in the image of the constant.
We collect all majority operations on~$\CarrierSet$ in the set
\m{\Maj\subs\Ops[3]} and all unary constants in the
set \m{\Consts=\lset{c_a}{a\in\CarrierSet}\subs\Ops[1]}.
\par
There is another important type of functions, which is not defined via
identities.
For $n\geq 2$ and $i\in n$ we call a function \m{f\in\Ops[n]} an
\emph{\nbdd{n}ary semiprojection on the position~$i$} if, for every
tuple \m{\mathff{x}=\apply{x_j}_{j\in n}\in \CarrierSet[n]} with
repetitions (i.e., there are $0\leq j<l<n$ with $x_j=x_l$), the result
\m{f\apply{\mathff{x}}} equals~\m{x_i}. An \nbdd{n}ary
semiprojection is an \nbdd{n}ary semiprojection on some
index~\m{i\in n}; we call it \emph{proper}, if it is not a projection.
Some authors restrict the
arities of semiprojections to be at least three; for us, a binary
semiprojection just subsumes an idempotent binary function, i.e., a
function satisfying \m{f(x,x)=x} for all \m{x\in\CarrierSet}.
\par

Moreover, we shall use \emph{permutations}, that is, bijective
self\dash{}maps \m{f\colon \CarrierSet\to\CarrierSet}. All permutations
on~$\CarrierSet$ form the symmetric group \m{\Sym(\CarrierSet)}. In
particular, \m{\Sym(n)} is the symmetric group of all permutations of
the \nbdd{n}element set \m{n=\set{0,\dotsc,n-1}}.

\subsection{Basic facts regarding commutation}\label{subsect:commutation}
As we want to study centralising monoids, we will be dealing with the
Galois correspondence between~$\Ops$ and~$\Ops[1]$ induced by
commutation (i.e., by the restriction of~$\commuteswith$ to
\m{\Ops\times\Ops[1]}). It is rather obvious from the definition that
in this case the commutation condition given in the introduction
simplifies as follows:
\begin{observation}\label{obs:commutation-unary}
For \m{n\in\N\setminus\set{0}}, \m{f\in\Ops[n]} and \m{s\in\Ops[1]} on
any set~$\CarrierSet$, we have \m{s\in\cent{\set{f}}} if and
only if the equality
\m{s\apply{f\apply{\mathff{x}}} = f\apply{s\circ \mathff{x}}}
holds for all \m{\mathff{x}\in\CarrierSet[n]}.
\end{observation}

The Galois closed sets are then, on the one side, all centralising
monoids
\[\lset{\Fn[1]{\cent{F}{}}}{F\subs\Ops}=
\lset{\Fn[1]{\bicent{S}{}}}{S\subs\Ops[1]},\]
and all centralisers of unary operations
\[\lset{\cent{S}}{S\subs\Ops[1]}
=\lset{\cent{\Fn[1]{\cent{F}{}}{}}}{F\subs\Ops}\]
on the other. Notably, a set \m{S\subs\Ops[1]} is a centralising
monoid, if and only if it is closed under the closure operator
\m{\Fn[1]{\bicent{\,}{}}}, so \m{\Fn[1]{\bicent{S}{}}=S}
(cf.~\cite[Definition/Lemma~2.2]{MachidaRosenbergCentralisingMonoids}).
A centralising monoid is \emph{maximal} if it is a co\dash{}atom (lower
cover of~$\Ops[1]$) in the lattice of Galois closures. Equivalently, a
centralising monoid~\m{M\subs\Ops[1]} is maximal if and only if its
centraliser~$\cent{M}$ is an atom above~\m{\cent{{\Ops[1]}}}. The
centraliser is therefore generated under the closure operator
\m{\cent{\Fn[1]{\cent{\,}{}}{}}} by any of its
non\dash{}trivial functions \m{f\in\cent{M}\setminus\cent{{\Ops[1]}}}.
As \m{\cent{M}=\cent{\Fn[1]{\cent{\set{f}}{}}{}}}, this means that
\m{M = \Fn[1]{\bicent{M}{}} =\set{f}^{*(1)**(1)}
                            =\Fn[1]{\cent{\set{f}}{}}}
for any \m{f\in\cent{M}\setminus\cent{{\Ops[1]}}}.
This condition entails an intimate connection between maximal
centralising monoids and generating functions of minimal clones as
classified in Rosenberg's Theorem~\ref{thm:Rosenberg-minimal-clones}.
This connection is known (see,
e.g.~\cite[Section~3, p.~155]{MachidaRosenbergCentralisingMonoids}),
and so we will only sketch the details. In particular, for the necessary
background facts about clones and their Galois theory, we have to refer
the reader to~\cite{PoeKal} or~\cite{LauFunctionAlgebrasOnFiniteSets}.

\newsavebox{\citationbox}
\sbox{\citationbox}{%
\bfseries cf.~\cite[Theorem~3.2]{MachidaRosenbergCentralisingMonoids}}
\begin{proposition}[{\usebox{\citationbox}}]%
\label{prop:generating-maximal-cent-monoids}
Given any maximal centralising monoid \m{M\subs\Ops[1]} on any finite
set~$\CarrierSet$, there is a minimal clone and a generating function
\m{f\in\cent{M}\setminus\cent{{\Ops[1]}}} of minimum arity of that
clone such that \m{M=\Fn[1]{\cent{\set{f}}{}}}.
\end{proposition}
\begin{proof}
For \m{\crd{\CarrierSet}<2} there are no maximal centralising monoids,
and the claim is trivial. For \m{\crd{\CarrierSet}\geq 2} the
centraliser \m{\cent{M}} is an atom strictly above
\m{\cent{{\Ops[1]}}\supseteq\TrivOps}. By finiteness of~$\CarrierSet$,
the centraliser \m{\cent{M}} contains some minimal clone~$F$. If~$F$
can be chosen such that it contains a minimum arity
generator~$f\notin\cent{{\Ops[1]}}$, we are done since any such~$f$
satisfies \m{M=\Fn[1]{\cent{\set{f}}{}}}. In the Boolean case,
\m{\cent{{\Op[1]{2}}}=\cent{\set{c_0,c_1,\neg}}} is the clone of
idempotent self\dash{}dual functions, so
\m{\cent{M}} contains the clone \m{\cent{\set{\neg}}\ni \neg} of
self\dash{}dual functions or the clone
\m{\cent{\set{c_0,c_1}}\ni\land,\lor} of idempotent operations.
Since~$\neg$ is self\dash{}dual but not idempotent, and the
Boolean lattice operations are idempotent but not self\dash{}dual, and
each of them generates a minimal clone, the claim holds for
\m{\crd{\CarrierSet}=2}. For \m{\crd{\CarrierSet}>2}, every function
\m{f\in\cent{{\Ops[1]}}} commutes with every unary operation and thus
preserves the kernel
\m{\ker(s) = \lset{(x,y)\in \CarrierSet[2]}{s(x)=s(y)}} of any unary
operation~$s$. Therefore, it preserves all equivalence relations
on~$\CarrierSet$, and from this one can show by induction on the arity
of~$f$ (the base case is the second exercise problem on p.~38
of~\cite{PierceTheoryOfAbstractAlgebras}, the inductive step is
contained in the `[d]irect proof of~2.2.' on p.~132
of~\cite{PoeLaengerRelationalSystemsWithTrivialEndPol}) that~$f$ must be a projection or constant (this requires
\m{\crd{\CarrierSet}>2}), cf.~\cite[3.3.~Example,
p.~136]{PoeLaengerRelationalSystemsWithTrivialEndPol} and the remark on p.~137, ibid.
Since \m{f\in\cent{{\Ops[1]}}\subs\cent{\Consts}}, it is idempotent and
thus fails to be a constant. So for \m{\crd{\CarrierSet}>2}, we have
\m{\cent{{\Ops[1]}}=\TrivOps}, and one may choose any minimal clone
\m{F\subs \cent{M}}, and any of its (minimum arity) generators~$f$.
\end{proof}

The preceding result says that every maximal centralising monoid on a
finite set has a singleton witness given by a minimum arity generator
of a minimal clone. Such functions are called \emph{minimal functions}
and were classified by Ivo Rosenberg~\cite{RosenbergMinimalClones}.

\begin{theorem}\label{thm:Rosenberg-minimal-clones}
Any minimal function~$f$ on a finite set~$\CarrierSet$ is of one of the
following types
\begin{enumerate}[(1)]
\item unary, and either a permutation of prime order or a
      non\dash{}permutation satisfying \m{f\circ f = f};
\item a ternary minority operation given as
      \m{f(x,y,z)=x\oplus y\oplus z} where
      \m{\algwops{\CarrierSet}{\oplus}} is an (Abelian) group such that
      every non\dash{}trivial element has order two;
\item a ternary majority operation;
\item a semiprojection of arity \m{2\leq k\leq \crd{\CarrierSet}}.
\end{enumerate}
\end{theorem}

It is by Rosenberg's Classification Theorem that
Proposition~\ref{prop:generating-maximal-cent-monoids} gains strength.
To find all maximal centralising monoids on a given set~$\CarrierSet$
one only has to iterate over all possible functions~$f$ listed in
Theorem~\ref{thm:Rosenberg-minimal-clones}, to compute
\m{\Fn[1]{\cent{\set{f}}{}}}, and to take the maximal proper submonoids
of~\m{\Ops[1]} from this list. This approach has been carried out
successfully for \m{\CarrierSet=\set{0,1,2}}
in~\cite{MachidaRosenbergCentralisingMonoids}, and it motivates to
study restrictions of the Galois connection given by commutation to
some type of function from Theorem~\ref{thm:Rosenberg-minimal-clones}
and unary functions, as it was done
in~\cite{MachidaRosenbergGoldsternCentralisingMonoids,MachidaRosenbergCentralisingMonoidsBinaryIdempotentOps},
again for \m{\crd{\CarrierSet}=3}.
\par

With the aim of pushing the work started
in~\cite{MachidaRosenbergCentralisingMonoids} to the next level, i.e.,
to \m{\CarrierSet=\set{0,1,2,3}}, we consider in this paper the Galois
correspondence given by commutation, restricted to
\m{\Maj\times\Ops[1]}.
The other cases, that is, semiprojections, minority functions and unary
operations will have to be treated separately.
\par

When studying centralising monoids with specific witnesses
\m{F\subs\Ops}, it is of course clear that
\m{\Fn[1]{\cent{F}{}} =\bigcap_{f\in F}\Fn[1]{\cent{\set{f}}{}}}, so
Observation~\ref{obs:commutation-unary} certainly applies here.
However, for \m{F\subs\Maj} the latter characterisation can be
simplified even more, as was already observed
in~\cite[Section~3, p.~207 et
seqq.]{MachidaRosenbergGoldsternCentralisingMonoids}. The
simplification is connected to the set \m{\sigma\subs\CarrierSet[3]}
of triples with three pairwise distinct entries, \ie,
\m{\sigma\defeq\lset{\apply{x,y,z}\in\CarrierSet[3]}{x\neq y\neq z\neq x}}.
Obviously, if \m{f\in\Maj} is a majority operation, \m{s\in\Ops[1]} and
\m{\mathff{x}\in \CarrierSet[3]\setminus\sigma}, then \m{s\circ
\mathff{x}\in\CarrierSet[3]\setminus\sigma}, too, and therefore the
condition \m{s\apply{f\apply{\mathff{x}}} = f\apply{s\circ \mathff{x}}} is trivially
fulfilled by the majority law. This implies the following observation:

\begin{observation}\label{obs:comm-maj-on-sigma}
For a majority operation \m{f\in\Maj} and any \m{s\in\Ops[1]} on some
set~$\CarrierSet$, we have \m{s\in\cent{\set{f}}} if and only if
\m{s\apply{f\apply{\mathff{x}}} = f\apply{s\circ \mathff{x}}} holds for all
\m{\mathff{x}\in\sigma}.
\end{observation}

In fact, this characterisation aligns nicely with the `relevant' part
of a majority operation: due to the majority law any majority operation
\m{f\in\Maj} is uniquely
determined by its restriction \m{f\Restriction_{\sigma}}. The number of
triples in \m{\sigma} amounts to \m{k\cdot (k-1)\cdot (k-2)} for a
\nbdd{k}element carrier set \m{\CarrierSet}. Hence, there exist
\m{k^{k\cdot(k-1)\cdot (k-2)}} such functions on a \nbdd{k}element set.
Even though this number is smaller than \m{\crd{\Ops[1]} = k^{k^3}},
we shall see in the next sections that it is critical for our task to
find more efficient ways of understanding the commutation Galois
correspondence than blindly iterating over \emph{all} majority
operations and using Observation~\ref{obs:comm-maj-on-sigma} to check
which of them commutes with which unary functions.

\section{Method}\label{sect:method}
It is possible to algorithmically enumerate the set (and moreover the lattice) of all Galois
closed sets of a Galois connection, if the inducing binary relation (in
our case \m{{\commuteswith}\subs\Maj\times\Ops[1]}) is
explicitly given as a cross table (that is, if the characteristic
function of the binary relation has been tabulated). This is one of the
corner stones of \emph{formal concept analysis}, see~\cite{GaWiFCA},
where such tables are referred to as \emph{formal contexts}. More
precisely, a formal context is a triple \m{\K=\apply{G,M,I}}
consisting of sets~$G$ of so\dash{}called \emph{objects} and~$M$ of
so\dash{}called \emph{attributes}, and an \emph{incidence relation}
\m{I\subs G\times M} inducing the Galois connection between them.
The known algorithms for producing the
collection~$\mathcal{L}$ of all Galois closed sets have a time
complexity of \m{\LandauO{\crd{G}\cdot\crd{M}^2}} per closed set (with
improvements under certain side conditions), see~\cite[Chapter 2, p.~39 et seqq.]{GanterObiedkovConceptualExploration}
for details. Thus, the total worst\dash{}case time complexity is given
by \m{\LandauO{\crd{\mathcal{L}}\cdot\crd{G}\cdot\crd{M}^2}}.
\par

For the present Galois connection the corresponding formal context
would be a $k^{k\cdot (k-1)\cdot (k-2)}\times k^k$ table, where~$k$ is
the cardinality of the carrier set~$\CarrierSet$. For
$k=4$, this is already a $4^{24}\times 256$ Boolean matrix, where the
enumeration of all $2\cdot 10^{14}<4^{24}<3\cdot 10^{14}$ rows is only
possible with very enhanced computational means. In fact, one would need
$256\,\mathrm{TB}$ to simply store it if one cell would take only one
bit (not byte; note that a standard integer nowadays normally takes
$64\,\mathrm{bit} = 8\,\mathrm{B}$), and the computation of the Galois
closures from such a table would become totally infeasible.
Thus, the main challenge that needs to be solved to produce a list of
all centralising monoids with majority witnesses on a
four\dash{}element set is to obtain a smaller context containing the
same information. Fortunately, it is typical that (especially when
there is a large discrepancy between the sizes of~$\crd{G}$
and~$\crd{M}$) formal contexts contain a lot of redundant bits, and we
shall see in Section~\ref{sect:results} that, for our problem, a much
smaller context suffices to enumerate all closed sets.
This is often due to the following two situations: if $g,h\in G$ are
incident with the same members of~$M$, i.e.,
\m{\set{g}'\defeq \lset{m\in M}{(g,m)\in I}} equals
\m{\set{h}'=\lset{m\in M}{(h,m)\in I}}, then only one of~$g$
and~$h$ needs to
be retained. In our example, this happens when two majority operations
commute with exactly the same unary operations. Identifying elements
of~$G$ when such equalities happen and reducing the context to a set of
representatives of the equivalence classes is known as \emph{object
clarification} in formal concept analysis. It can, of course, be dually
performed for~$M$, in which case one speaks of \emph{attribute
clarification}. Moreover, it can happen that some object \m{g\in G}
satisfies \m{\set{g}' = H' \defeq \bigcap_{h\in H}\set{h}'} for some
set \m{H\subs G\setminus\set{g}}. Then, the row of the context
belonging to~\m{g} can be dispensed with if the elements in~$H$ are
kept. Seen from a more abstract point of view, the closure \m{\set{g}'}
is \nbdd{\bigcap}reducible in the lattice of all Galois closures, and
so one only needs to keep elements \m{g\in G}
belonging to \nbdd{\bigcap}irreducible closures. Purifying a formal
context by removing \nbdd{\bigcap}reducible objects from~$G$ is known
as \emph{object reduction}, the dual operation with respect to~$M$ as
\emph{attribute reduction}.
\par

For the problem of determining centralising monoids, we shall
understand the set of majority operations on~\m{\CarrierSet} as
\emph{objects}~$G$ and the set of unary operations~\m{\Ops[1]} as
\emph{attributes}~$M$. The roles of objects and attributes are, of
course, completely arbitrary. Since~\m{G} is much bigger than~\m{M} in
our case, attribute clarification and reduction is much less important
than object clarification (and reduction). However, as it turns out, a
somewhat combined approach will be the key to success.
For every \m{s\in M=\Ops[1]} we shall characterise (and subsequently
enumerate) \m{\set{s}' = \cent{\set{s}}\cap \Maj}, and for every \m{f\in
\set{s}'} we shall compute and store
\m{\set{f}' = \Ops[1]\cap\cent{\set{f}}}. This may seem like an overly
convoluted approach to just obtain some rows of the context; however,
as we shall show in Section~\ref{sect:results}, when putting
all \m{\set{f}'} from all \m{s\in\Ops[1]} together, these are
sufficiently many rows to give us the full picture. In fact, by this
procedure, we shall only miss the rows \m{\set{f}'} for
those~\m{f\in\Maj} that do not commute with any unary operation besides
the trivial ones, i.e., those that commute with every majority
operation. One such row can simply be added manually to the context.
\par
So far, there is no guarantee that we will consider substantially less
majority operations than all \m{k^{k(k-1)(k-2)}} (i.e.\ \m{4^{24}} for
\m{k=4}) by this approach. Notably,
we have to identify (and ignore) those (trivial) unary operations that
commute with every majority operation. In general, it
will be critical to observe when we characterise \m{\set{s}'} for each
unary \m{s\in\Ops[1]} that the size of \m{\set{s}'} is not too big,
where `not too big' means feasible to be enumerated in a reasonable
amount of time with practical resources. As it turns out, for
\m{\CarrierSet = \set{0,1,2,3}}, we will be lucky in this respect.
\par
As a side product of considering and characterising \m{\set{s}'} for
each (non\dash{}trivial) \m{s\in\Ops[1]} individually, we shall also
perform some attribute clarifications, although for one portion of the
unary maps we have not executed clarifications to the full extent as it
is not necessary for our task.

\section{Characterisation of commutation and attribute clarification}%
\label{sect:attribute-clarification}%

We proceed, as explained in Section~\ref{sect:method}, by characterising
\m{\cent{\set{s}}\cap\Maj} for each \m{s\in\Ops[1]} over \m{\CarrierSet
= \set{0,1,2,3}}. In this respect we
partition the set of attributes, \m{\Ops[1]}, according to the number
of elements in the image of its members. Moreover, the attributes with
full image, \ie\ permutations, will be partitioned further, depending
on the number of fixed points. It was already observed by
Harnau~\cite[p.~340 et seqq.]{Harnau1DefVertauschbarkeitMaximalitaet}
that the latter is an important quantitative aspect of unary functions
in connection with commutation.
\par

\subsection{One-element image}\label{subsect:constants}%
A function \m{s\in\Ops[1]} such that \m{\crd{\im\apply{s}}=1} is
constant, and therefore automatically commutes with any idempotent
operation, and any majority operation in particular. Namely, if
\m{f\apply{x,\dotsc,x}\approx x}, then~$f$ preserves~$\set{a}$ for
every \m{a\in\CarrierSet}, and preserving a singleton relation~$\set{a}$
is equivalent to commuting with the corresponding unary constant
operation~$c_a$. Furthermore, projections commute with any
operation, which holds true for the identity operation, in particular.
Thus, we always get
\m{\Consts\cup\set{\id_{\CarrierSet}}\subs \cent{\set{f}}} for idempotent
operations~$f$, and so all constant operations will be identified with
the identity operation by attribute clarification (and be ignored
henceforth).  It will become clear from the characterisations given in
the other subsections that apart from the identity and constants no
other unary operation commutes with all majority operations,
i.e.~\m{\Op[1]{4}\cap\cent{\Maj[4]} =
\set{\id_{4},c_0,c_1,c_2,c_3}}.

\subsection{Two-element image}\label{subsect:two-element-image}%
In this subsection we consider a function \m{s\in\Ops[1]} such that
\m{\im\apply{s}=\set{\alpha,\beta}} with \m{\alpha\neq\beta}. Suppose that
\m{k=\crd{\CarrierSet}}. We present two lemmata, the first one regarding
\m{\crd{s^{-1}\fapply{\set{\alpha}}}=k-1}, the other one concerning
\m{\crd{s^{-1}\fapply{\set{\alpha}}}=k-2}. For \m{k=4} this describes the
complete picture.
\par

\sbox{\citationbox}{%
\bfseries cf.~\cite[Lemma~3.2]{MachidaRosenbergGoldsternCentralisingMonoids}}
\begin{lemma}[{\usebox{\citationbox}}]%
\label{lem:im=2--k-1}
Let \m{f\in\Ops[3]} be a majority operation on \m{\CarrierSet} with
\m{\crd{\CarrierSet}=k} and let \m{s\in\Ops[1]} be such that
\m{\im\apply{s} = \set{\alpha,\beta}},
\m{\crd{s^{-1}\fapply{\set{\alpha}}}=k-1}
and \m{\crd{s^{-1}\fapply{\set{\beta}}}=1}.
Then \m{f\commuteswith s} if and only if
\m{f\apply{\mathff{x}}\in s^{-1}\fapply{\set{\alpha}}}
for all \m{\mathff{x}\in\sigma}.
\end{lemma}
\begin{proof}
By Observation~\ref{obs:comm-maj-on-sigma} the commutation condition only
has to be checked on~\m{\sigma}. If \m{\apply{a,b,c}\in\sigma}, then at
least two of \m{s\apply{a}}, \m{s\apply{b}} and \m{s\apply{c}} are equal to
\m{\alpha} (because either one or none of \m{a}, \m{b} and \m{c} belongs to
\m{s^{-1}\fapply{\set{\beta}}}). So,
\m{f\apply{s\apply{a},s\apply{b},s\apply{c}} =\alpha}, whence the
commutation condition becomes
\m{s\apply{f\apply{a,b,c}} =
f\apply{s\apply{a},s\apply{b},s\apply{c}}=\alpha} for all
\m{\apply{a,b,c}\in\sigma}.
\end{proof}

For \m{k=4} there are four possibilities to choose the singleton
\m{s^{-1}\fapply{\set{\beta}}} as in Lemma~\ref{lem:im=2--k-1}; the
preimage \m{s^{-1}\fapply{\set{\alpha}}}, which completely determines the
condition for commutation, is then the complement.
We list the four respective conditions for \m{\CarrierSet=\set{0,1,2,3}}
below:
\begin{align}
f\apply{\mathff{x}}\in\set{1,2,3}
\text{ for all } \mathff{x}\in\sigma
\tag{A1}\\
f\apply{\mathff{x}}\in\set{0,2,3}
\text{ for all } \mathff{x}\in\sigma
\tag{A2}\\
f\apply{\mathff{x}}\in\set{0,1,3}
\text{ for all } \mathff{x}\in\sigma
\tag{A3}\\
f\apply{\mathff{x}}\in\set{0,1,2}
\text{ for all } \mathff{x}\in\sigma
\tag{A4}
\end{align}
\par

In order to clarify the context, we shall replace all the equivalent unary
functions \m{s\in\Ops[1]} with exactly two distinct values one of which is
attained on a three-element set \m{B} by the condition (A\m{i}) involving
\m{B}.
\par

Note that, even though the characterising condition given
in Lemma~\ref{lem:im=2--k-1} is the easiest non\dash{}trivial
one in this paper, it is the most time consuming from the computational
perspective.
Namely, $(k-1)^{\crd{\sigma}} = (k-1)^{k(k-1)(k-2)}$ majority
operations need to be enumerated per unary operation, which for \m{k=4}
is \m{3^{24}=282\,429\,536\,481}.
\par

Next we consider the other case that can happen for unary functions with a
two\dash{}element range on a four\dash{}element domain. Here
\m{2^{3!(k-2)}\cdot
(k-2)^{\crd{\sigma}-3!(k-2)}}, that is,
\m{2^{6(k-2)}\cdot(k-2)^{(k-2)(k+2)(k-3)}}, majority operations need to
be considered per unary function. For \m{k=4} this amounts
to only~\m{2^{24}=16\,777\,216} functions.
\begin{lemma}\label{lem:im=2--k-2}
Let \m{f\in\Ops[3]} be a majority operation on \m{\CarrierSet} with
\m{\crd{\CarrierSet}=k} and let \m{s\in\Ops[1]} be such that
\m{\im\apply{s} = \set{\alpha,\beta}}, \m{\alpha\neq\beta},
\m{\crd{s^{-1}\fapply{\set{\alpha}}}=k-2}
and \m{\crd{s^{-1}\fapply{\set{\beta}}}=2}.
Then \m{f\commuteswith s} holds if and only if
\m{f\apply{a,b,c}\in s^{-1}\fapply{\set{\beta}}}
for all triples \m{\apply{a,b,c}\in\sigma} such that
\m{s^{-1}\fapply{\set{\beta}}\subs\set{a,b,c}} and
\m{f\apply{a,b,c}\in s^{-1}\fapply{\set{\alpha}}} else.
\end{lemma}
\begin{proof}
As in the proof of Lemma~\ref{lem:im=2--k-1}, it suffices to consider
the commutation condition on triples \m{\apply{a,b,c}\in\sigma}. If
\m{\set{a,b,c}\sups s^{-1}\fapply{\set{\beta}}}, then
\m{f\apply{s\apply{a},s\apply{b},s\apply{c}}=\beta}; otherwise, we have
\m{\crd{\set{a,b,c}\cap s^{-1}\fapply{\set{\beta}}}\leq 1}, so
\m{\crd{\set{a,b,c}\cap s^{-1}\fapply{\set{\alpha}}}\geq 2}, which implies
\m{f\apply{s\apply{a},s\apply{b},s\apply{c}}=\alpha}. Therefore, for all
triples \m{\apply{a,b,c}\in\sigma} satisfying that
\m{s^{-1}\fapply{\set{\beta}}\subs\set{a,b,c}}, the commutation condition
becomes \m{s\apply{f\apply{a,b,c}}=\beta}; for all other triples it turns
into \m{s\apply{f\apply{a,b,c}}=\alpha}.
\end{proof}

For the case that \m{k=4}, the characterisation in
Lemma~\ref{lem:im=2--k-2} can be rephrased a little bit more concretely.
\begin{corollary}\label{cor:im=2--k-2}
Suppose that \m{\CarrierSet = \set{a,b,c,d}}, where
\m{\crd{\CarrierSet}=4}, \m{f\in\Ops[3]} is a majority function and
\m{s\in\Ops[1]} satisfies \m{\im\apply{s}=\set{\alpha,\beta}} with
\m{\alpha\neq\beta},
\m{s^{-1}\fapply{\set{\alpha}}=\set{a,b}}
and \m{s^{-1}\fapply{\set{\beta}}=\set{c,d}}.
Then \m{f\commuteswith s} if and only if \m{f\apply{x,y,z}\in\set{c,d}} for
all \m{\apply{x,y,z}\in\sigma} such that \m{\set{c,d}\subs\set{x,y,z}} and
\m{f\apply{x,y,z}\in\set{a,b}} for all \m{\apply{x,y,z}\in\sigma} such that
\m{\set{a,b}\subs\set{x,y,z}}.
\end{corollary}
\begin{proof}
For every triple \m{\apply{x,y,z}\in\sigma} where
\m{\set{c,d}\not\subs\set{x,y,z}}, the set \m{\set{x,y,z}} has to miss one
of \m{c} or \m{d}. In order to contain three distinct elements on
\m{\CarrierSet=\set{a,b,c,d}}, it has to contain the other one, as well as
both elements \m{a} and \m{b}. For this reason the claim follows from
Lemma~\ref{lem:im=2--k-2}.
\end{proof}

For \m{\CarrierSet=\set{0,1,2,3}}, Corollary~\ref{cor:im=2--k-2} describes exactly
\m{\binom{4}{2}/2=3} concrete conditions:
\begin{align}
&f\apply{x,y,z}\in\set{0,1}
\text{ for all }\apply{x,y,z}\in\sigma \text{ such that }
\set{0,1}\subs\set{x,y,z}
\text{ and }\notag{}\\
&f\apply{x,y,z}\in\set{2,3}
\text{ for all }\apply{x,y,z}\in\sigma \text{ such that }
\set{2,3}\subs\set{x,y,z}
\tag{A5}\\
&f\apply{x,y,z}\in\set{0,2}
\text{ for all }\apply{x,y,z}\in\sigma \text{ such that }
\set{0,2}\subs\set{x,y,z}
\text{ and }\notag{}\\
&f\apply{x,y,z}\in\set{1,3}
\text{ for all }\apply{x,y,z}\in\sigma \text{ such that }
\set{1,3}\subs\set{x,y,z}
\tag{A6}\\
&f\apply{x,y,z}\in\set{0,3}
\text{ for all }\apply{x,y,z}\in\sigma \text{ such that }
\set{0,3}\subs\set{x,y,z}
\text{ and }\notag{}\\
&f\apply{x,y,z}\in\set{1,2}
\text{ for all }\apply{x,y,z}\in\sigma \text{ such that }
\set{1,2}\subs\set{x,y,z}
\tag{A7}
\end{align}
We shall use these to replace unary functions \m{s} with exactly two
values \m{\alpha} and \m{\beta}, whose kernel partition is
\m{\set{\set{0,1},\set{2,3}}}, \m{\set{\set{0,2},\set{1,3}}} and
\m{\set{\set{0,3},\set{1,2}}}, respectively.
\par

\subsection{Three-element image}\label{subsect:three-element-image}%
For unary functions on a four\dash{}element domain with
three\dash{}element range the analysis of the commutation condition
needs to go a bit more into detail to be satisfactory.
There are \m{\binom{4}{3}\cdot\binom{3}{1}\cdot\binom{4}{2}\cdot 2!=144} functions
of this type as there are \m{\binom{4}{3}} ways to choose the image,
\m{\binom{3}{1}} ways to choose which element in the image has a
two\dash{}element preimage, \m{\binom{4}{2}} ways to choose
this preimage, and \m{2!} possibilities to arrange the remaining values.
\par

\begin{lemma}\label{lem:s-im-3-12+12-tuples}
Assume \m{\crd{\CarrierSet}=4} and let \m{s\in\Ops[1]} with
\m{\crd{\im{s}} = 3}, \m{\im{s} = \set{\alpha,\beta,\gamma}},
\m{\CarrierSet = \set{u,v,x,y}}
be such that \m{s(u) =s(v) = \alpha}, \m{s(x)=\beta}, \m{s(y)=\gamma}.
Let \m{t\in\CarrierSet} be such that
\m{\CarrierSet=\set{\alpha,\beta,\gamma,t}} and consider an arbitrary
triple \m{(a,b,c)\in\sigma}.
If \m{t\in\set{a,b,c}}, then there is no \m{\mathff{x}\in\CarrierSet[3]} such
that \m{s\circ\mathff{x}=\apply{a,b,c}}.
If, otherwise, \m{t\notin\set{a,b,c}}, then
\m{(a,b,c)=(\alpha,\beta,\gamma)\circ\pi} for some index
permutation $\pi\in \Sym(3)$ and we have a two\dash{}element preimage
\m{\lset{\mathff{x}\in\CarrierSet[3]}{s\circ\mathff{x}=(a,b,c)}
=\set{(u,x,y)\circ\pi, (v,x,y)\circ\pi}}.
\end{lemma}
\begin{proof}
If there is $\mathff{x}\in\CarrierSet[3]$ with
\m{s\circ\mathff{x}=(a,b,c)}, then
\m{\set{a,b,c}\subs\im{s}=\set{\alpha,\beta,\gamma}},
which equals \m{\CarrierSet\setminus\set{t}}, so \m{t\notin \set{a,b,c}}.
If this happens, then
\m{\set{a,b,c}} is a subset of
\m{\CarrierSet\setminus\set{t}=\set{\alpha,\beta,\gamma}}, and, as both
sets have three elements, they must coincide. Thus, there is
\m{\pi\in\Sym(3)} such that
\[(a,b,c)=(\alpha,\beta,\gamma)\circ\pi=
(s(u),s(x),s(y))\circ\pi = (s(v),s(x),s(y))\circ\pi.\]
Since \m{s^{-1}\fapply{\set{\alpha}}=\set{u,v}},
\m{s^{-1}\fapply{\set{\beta}}=\set{x}} and
\m{s^{-1}\fapply{\set{\gamma}}=\set{y}}, there are no other triples with
this property.
\end{proof}

\begin{observation}\label{obs:tuples-determining-other-tuples}
Consider \m{s\in\Ops[1]} with \m{\crd{\im{s}} = 3}, \m{\im{s} =
\set{\alpha,\beta,\gamma}}; let $f\in\Ops[3]$ and \m{(a,b,c)\in\sigma}.
Under the condition \m{s\commuteswith f}, the value \m{f(a,b,c)} only
determines values \m{f(\mathff{x})} for some other
\m{\mathff{x}\in\CarrierSet[3]}, if \m{\mathff{x}\in\sigma} and
\m{\set{a,b,c}=\im{s}=\set{\alpha,\beta,\gamma}}.
\end{observation}
\begin{proof}
The condition \m{s\commuteswith f} may determine \m{f(\mathff{x})} via
\m{s(f(\mathff{x}))=f(s\circ\mathff{x})}, that is, by
\m{f(\mathff{x})\in s^{-1}\fapply{f(s\circ\mathff{x})}}. For $f(a,b,c)$
to have an influence on~$f(\mathff{x})$, we need to have
\m{s\circ\mathff{x}=(a,b,c)}, i.e., \m{\mathff{x}\in\sigma} and
\m{\set{a,b,c}\subs\im{s}=\set{\alpha,\beta,\gamma}}. As both sets have
three distinct elements, they must coincide.
\end{proof}

\begin{lemma}\label{lem:s-with-im3}
Assume \m{\crd{\CarrierSet}=4} and let \m{s\in\Ops[1]} with
\m{\crd{\im{s}} = 3}, \m{\im{s} = \set{\alpha,\beta,\gamma}},
\m{\CarrierSet = \set{u,v,x,y}}
be such that \m{s(u) =s(v) = \alpha}, \m{s(x)=\beta}, \m{s(y)=\gamma}.
Moreover, let \m{f\in\Ops[3]} be a majority operation.
Now, \m{f\commuteswith s} if and only if
\begin{align*}
f((u,v,z)\circ\pi) &\in s^{-1}\fapply{\set{\alpha}}=\set{u,v}\\
f((w,x,y)\circ\pi) &\in
s^{-1}\fapply{\set{f((\alpha,\beta,\gamma)\circ\pi)}}
\end{align*}
hold for all \m{z\in\set{x,y}}, \m{w\in\set{u,v}} and all
\m{\pi\in\Sym(3)}.
\end{lemma}
\begin{proof}
By Observation~\ref{obs:comm-maj-on-sigma}, we have
\m{f\commuteswith s} if and only if \m{f(\mathff{x})\in
s^{-1}\fapply{\set{s\circ\mathff{x}}}} holds for all
\m{\mathff{x}\in\sigma}. Every tuple \m{\mathff{x}\in \sigma}
consists of three distinct entries, wherefore
\m{\im{\mathff{x}}\in\set{\set{u,v,x},\set{u,v,y},\set{u,x,y},\set{v,x,y}}}.
If \m{\mathff{x}=(u,v,z)\circ\pi} with \m{z\in\set{x,y}} and some
\m{\pi\in\Sym(3)}, then the above condition becomes
\begin{align*}
f(\mathff{x})=f((u,v,z)\circ\pi)\in
s^{-1}\fapply{\set{f(s\circ(u,v,z)\circ\pi)}}
&=s^{-1}\fapply{\set{f((\alpha,\alpha,s(z))\circ\pi)}}\\
&=s^{-1}\fapply{\set{\alpha}}=\set{u,v}.
\end{align*}
If \m{\mathff{x}=(w,x,y)\circ\pi} for some \m{\pi\in\Sym(3)} and
\m{w\in\set{u,v}}, the condition turns into
\begin{align*}
f(\mathff{x})=f((w,x,y)\circ\pi)\in s^{-1}\fapply{\set{f(s\circ(w,x,y)\circ\pi)}}
&=s^{-1}\fapply{\set{f((\alpha,\beta,\gamma)\circ\pi)}}.
\end{align*}
Therefore, the characterisation in the lemma is obtained from
Observation~\ref{obs:comm-maj-on-sigma} by partitioning~$\sigma$ into
two times two classes of \m{3!=6} tuples each.
\end{proof}

\begin{corollary}\label{cor:s-with-im3}
Assume \m{\crd{\CarrierSet}=4} and let \m{s\in\Ops[1]} with
\m{\crd{\im{s}} = 3}, \m{\im{s} = \set{\alpha,\beta,\gamma}},
\m{\CarrierSet = \set{u,v,x,y}}
be such that \m{s(u) =s(v) = \alpha}, \m{s(x)=\beta}, \m{s(y)=\gamma}.
The restriction
\m{\zeta\defeq
s\restriction_{\set{\alpha,\beta,\gamma}}\in\Op[1]{\set{\alpha,\beta,\gamma}}}
satisfies
\m{\zeta\notin \Sym\apply{\set{\alpha,\beta,\gamma}}
\iff
\set{u,v}\subs\set{\alpha,\beta,\gamma}}.
\par
Moreover, let \m{f\in\Ops[3]} be a majority operation.
\begin{enumerate}[(a)]
\item\label{item:zeta-not-sym}
      If \m{\zeta\notin\Sym\apply{\set{\alpha,\beta,\gamma}}}, let the
meaning of the symbols
      \m{(x,\beta)} and \m{(y,\gamma)} be chosen such that
      $\set{\alpha,\beta,\gamma}=\set{u,v,x}$, i.e.,
      $(\alpha,\beta,\gamma)=(u,v,x)\circ\xi$
      for some $\xi\in\Sym(3)$.
      \par
      In this case, we have \m{f\commuteswith s} if and only if
\begin{align*}
f((u,v,z)\circ\pi) &\in s^{-1}\fapply{\set{\alpha}}=\set{u,v}\\
f((w,x,y)\circ\pi) &\in
s^{-1}\fapply{\set{f((\alpha,\beta,\gamma)\circ\pi)}}
=s^{-1}\fapply{\set{f((u,v,x)\circ\xi\circ\pi)}}
\end{align*}
      hold for all \m{z\in\set{x,y}}, \m{w\in\set{u,v}} and
      \m{\pi\in\Sym(3)}.
\item\label{item:zeta-in-sym}
      If $\zeta\in\Sym\apply{\set{\alpha,\beta,\gamma}}$, choose the
      symbols~$u$ and~$v$ such that
      \m{v\notin\set{\alpha,\beta,\gamma}\ni u}, i.e.,
      \m{\set{\alpha,\beta,\gamma}=\set{u,x,y}}, so there is
      \m{\xi\in\Sym(3)} with
      \m{(\alpha,\beta,\gamma)=(u,x,y)\circ\xi}.
      \par
      In this case, we have \m{f\commuteswith s} if and only if
\begin{align*}
f((u,v,z)\circ\pi) &\in s^{-1}\fapply{\set{\alpha}}=\set{u,v}\\
f((w,x,y)\circ\pi) &\in
s^{-1}\fapply{\set{f((\alpha,\beta,\gamma)\circ\pi)}}
=s^{-1}\fapply{\set{f((u,x,y)\circ\xi\circ\pi)}}
\end{align*}
      hold for all \m{z\in\set{x,y}}, \m{w\in\set{u,v}} and
      \m{\pi\in\Sym(3)}.
\end{enumerate}
\end{corollary}
\begin{proof}
First we argue that \m{\zeta\notin\Sym\apply{\set{\alpha,\beta,\gamma}}} if and
only if \m{\set{u,v}\subs\set{\alpha,\beta,\gamma}}. Indeed, if~$\zeta$
is not a permutation, it cannot be injective, and hence must send two
distinct arguments to the same value. By the structure of~$s$, these
arguments must be \m{u} and \m{v}, so these values must occur in the
domain of~$\zeta$, which is \m{\set{\alpha,\beta,\gamma}}. The converse
is evident.
\par
The two characterisations are now immediate consequences of
Lemma~\ref{lem:s-with-im3} and of the choice which elements
of~\m{\CarrierSet} are denoted by \m{x,\beta,y,\gamma} and \m{u,v},
respectively.
\end{proof}

Note that, according to
Observation~\ref{obs:tuples-determining-other-tuples}, any choice of the
value \m{f((u,v,y)\circ\pi)} in the first part of the condition in
Corollary~\ref{cor:s-with-im3}\eqref{item:zeta-not-sym} does not
influence other values of~$f$, whereas a choice of
\m{f((u,v,x)\circ\pi)}, which can be
arbitrary in \m{s^{-1}\fapply{\set{\alpha}}}, limits the choices of
\nbdd{f}values for exactly two other triples in~\m{\sigma} (as predicted
by Lemma~\ref{lem:s-im-3-12+12-tuples}).
\par
By the same logic, the twelve values \m{f((u,v,z)\circ\pi)} in the first
part of the condition in
Corollary~\ref{cor:s-with-im3}\eqref{item:zeta-in-sym} can be chosen
freely in \m{s^{-1}\fapply{\set{\alpha}}}, but the values of~\m{f}
on permutations of~\m{(v,x,y)} depend on the choice of~\m{f} on
some (other) permutation of~\m{(u,x,y)}, which in turn has cyclic
interdependencies with \nbdd{f}values on other permutations
of~\m{(u,x,y)}. The details of these interrelations depend on the action
of~$s$ (more precisely, of the cyclic group
$S=\gapply{\set{\zeta}}_{\Sym\apply{\set{\alpha,\beta,\gamma}}}$) on the
set~\m{\sigma'} of six permutations of \m{(u,x,y)} (or
\m{(\alpha,\beta,\gamma)}), namely on the index
map~$\xi\in\Sym(3)$ that is uniquely determined by
$s\circ(u,x,y)=(\alpha,\beta,\gamma)=(u,x,y)\circ\xi$.
The action of~$S$ on~\m{\sigma'=\lset{(u,x,y)\circ\pi}{\pi\in\Sym(3)}}
partitions~$\sigma'$ into $6/n$ orbits of size~$n\in\set{1,2,3}$
where~$n$ is the order of~$\zeta$. For each of these orbits, the choice
of~$f$ on only one point of the orbit uniquely determines the values
of~$f$ on all triples in the orbit (see also
Subsection~\ref{subsect:permutation} for more explanation), and these
values restrict the values of~$f$ on the corresponding triples
$(v,x,y)\circ\pi$ as demonstrated in Corollary~\ref{cor:s-with-im3}.
The values of~$f$ on triples belonging to different orbits do not
influence each other with respect to \m{s\commuteswith f}.
For more clarity we give an explicit example.
\par

\begin{example}\label{ex:3-element-image}
Let \m{s\in\Op[1]{4}} be given by \m{s\circ(0,1,2,3)=(3,0,0,1)}. Then
we have \m{\im{s}=\set{0,1,3}} and \m{\alpha=0} since
\m{B=s^{-1}\fapply{\set{0}}=\set{1,2}}.
Hence, \m{v=\nolinebreak2\in B\setminus\im{s}} and
$u=1\in B\setminus\set{v}$. As
$\{x,y\}=\{0,1,2,3\}\setminus B=\nolinebreak\{0,3\}$,
we choose \m{x=0} and \m{y=3}, and thus \m{\beta=s(x)=3} and
\m{\gamma=s(y)=1}.
Now \m{\zeta=s\restriction_{\set{0,1,3}} = (0\,3\,1)}
is a cyclic shift of order three and thus
\[
s\circ(u,x,y)=s\circ(1,0,3)=(0,3,1)=(1,0,3)\circ\xi
\]
is also given as a cyclic index shift \m{\xi=(0\,1\,2)\in\Sym(3)}.
The six permutations of $(u,x,y)=(1,0,3)$ are partitioned into two
orbits of size~$3$ under the action of
\m{\gapply{\set{\zeta}}_{\Sym(\set{0,1,3})}}, and the action of~$s$
(shown as arrows) on permutations of \m{(v,x,y)=(2,0,3)} is connected
to these orbits as follows:
\begin{center}
\begin{tikzpicture}[x=1em,y=1em,
pkt/.style={shape=circle,draw,minimum size=3pt,
            inner sep=0pt,outer sep=0.5\pgflinewidth}]
\begin{scope}[scale=2]
\node[pkt,label={below:{$\scriptstyle(1,0,3)$}}] (1) at (0,0){};
\node[pkt,label={below:{$\scriptstyle(0,3,1)$}}] (2) at (2,0){};
\node[pkt,label={above:{$\scriptstyle(3,1,0)$}}] (3) at (1,1.732){};
\node[pkt,label={above:{$\scriptstyle(3,2,0)$}}] (4) at (-1,0){};
\node[pkt,label={above:{$\scriptstyle(2,0,3)$}}] (5) at (3,0){};
\node[pkt,label={below:{$\scriptstyle(0,3,2)$}}] (6) at (2,1.732){};
\draw[->,>=stealth] (4)-> (1);
\draw[->,>=stealth] (1)-> (2);
\draw[->,>=stealth] (2)-> (3);
\draw[->,>=stealth] (3)-> (1);
\draw[->,>=stealth] (5)-- (2);
\draw[->,>=stealth] (6)-- (3);
\end{scope}
\begin{scope}[scale=2,shift={(6,0)}]
\node[pkt,label={below:{$\scriptstyle(1,3,0)$}}] (1) at (0,0){};
\node[pkt,label={below:{$\scriptstyle(0,1,3)$}}] (2) at (2,0){};
\node[pkt,label={above:{$\scriptstyle(3,0,1)$}}] (3) at (1,1.732){};
\node[pkt,label={above:{$\scriptstyle(3,0,2)$}}] (4) at (-1,0){};
\node[pkt,label={above:{$\scriptstyle(2,3,0)$}}] (5) at (3,0){};
\node[pkt,label={below:{$\scriptstyle(0,2,3)$}}] (6) at (2,1.732){};
\draw[->,>=stealth] (4)-> (1);
\draw[->,>=stealth] (1)-> (2);
\draw[->,>=stealth] (2)-> (3);
\draw[->,>=stealth] (3)-> (1);
\draw[->,>=stealth] (5)-- (2);
\draw[->,>=stealth] (6)-- (3);
\end{scope}
\end{tikzpicture}
\end{center}
By the preimage condition
$f((w,x,y)\circ\pi)\in s^{-1}\fapply{\set{f((u,x,y)\circ\xi\circ\pi)}}$
for $w\in\set{u,v}$ and \m{\pi\in\Sym(3)}
in Corollary~\ref{cor:s-with-im3}\eqref{item:zeta-in-sym}, we have to choose
\nbdd{f}values in \nbdd{s}preimages, that is in $\set{1,2}$, $\set{0}$
or $\set{3}$, in the opposite direction of the arrows given by the
action of~$s$ in order to ensure \m{f\commuteswith s}.
\par
For example, on any permutation of $(u,x,y)=(1,0,3)$ we are not
allowed to choose the value~$2$, as otherwise the \nbdd{s}preimage
would be empty and there would be no possible \nbdd{f}value
for a (here and generally different) permutation of
\m{(v,x,y)=(2,0,3)}. In fact, we pick an arbitrary representative from
each orbit of $(1,0,3)$, for instance, $(1,0,3)$ and $(1,3,0)$,
independently select an arbitrary \nbdd{f}value in $\set{0,1,3}$ for
it, and this determines the \nbdd{f}values for all permutations
of~$(1,0,3)$ uniquely (since $2$ is a forbidden value).
After that the values for the permutations of $(v,x,y)=(2,0,3)$ can be
chosen freely in the respective \nbdd{s}preimage of these six unique
values. Unconstrained by this, the \nbdd{f}values for triples
from~$\sigma$ containing~$1$ and~$2$ can be fixed arbitrarily
in~\m{\set{u,v}=\set{1,2}} under the condition \m{s\commuteswith f}.
\par
To be concrete, if we select, e.g., \m{f(1,0,3)=1} and \m{f(1,3,0)=0},
then
\begin{align*}
f(3,1,0)&=3 & f(0,3,1)&=0\\
f(3,2,0)&=3 & f(0,3,2)&=0& f(2,0,3)&\in\set{1,2}\\
f(3,0,1)&=1 & f(0,1,3)&=3\\
f(3,0,2)&\in\set{1,2}& f(0,2,3)&=3 & f(2,3,0)&=0.
\end{align*}
Moreover, we have to combine this with any of the $2^{12}$ independent
choices for \m{f((1,2,z)\circ\pi)\in\set{1,2}} where \m{z\in\set{0,3}}
and \m{\pi\in\Sym(3)}.
\end{example}

In the light of the preceding interpretation, the description in
Corollary~\ref{cor:s-with-im3} is sufficient to computationally
enumerate all majority operations in~\m{\cent{\set{s}}} for each of the~$144$
unary operations~$s\in\Ops[1]$ with a three\dash{}element image on
\m{\CarrierSet = \set{0,1,2,3}}.
These~$144$ cases fall into finitely many classes depending on the
unique element in~$\CarrierSet\setminus\im{s}$, the choice of the unique
two\dash{}element preimage $\set{u,v}$, the fact whether
$s\restriction_{\im{s}}$ is a permutation or not, and, if so, which
order it has, and some more subtle conditions regarding the roles
of~$u,v$ (in the permutation case) and of~$x,\beta,y,\gamma$ (in the
non\dash{}permutation case). In the context of our task, the exact
number of classes is not relevant enough, so that we forgo an
explicit listing of the conditions as in
Subsections~\ref{subsect:two-element-image} and~\ref{subsect:permutation}.
It is more important for us to note that a given
three\dash{}valued~$s\in\Ops[1]$ sufficiently restricts the search
space of majority operations~$f$, so that all possible candidates
$f\in\cent{\set{s}}$ can be enumerated (and then checked) in a reasonable
amount of time. A rough upper bound can be derived by observing that
preimages of any given~$s$ have size at most two, and that there are
only~$24$ relevant triples in~$\sigma$. Hence, according to the
conditions presented in Corollary~\ref{cor:s-with-im3}, at
most~$2^{24}=16\,777\,216$ majority operations (and with more
programming skills even much less, due to some preimages being
singletons) need to be considered.

\subsection{Permutations}\label{subsect:permutation}%
Before we start with a more detailed analysis,
let us note the following trivial observation,
which follows from the fact that every permutation on a finite set has
a finite order and centralisers are closed under composition and thus
under taking finite powers.
\begin{observation}\label{obs:inverse-perm}
For a permutation $s\in\Sym\apply{\CarrierSet}$ on a finite
set~$\CarrierSet$ we have \m{s\commuteswith f} if and only
\m{s^{-1}\commuteswith f} for any finitary operation \m{f\in\Ops}. In
other words \m{\cent{\set{s}} = \cent{\set{s^{-1}}}}, whence \m{s} and \m{s^{-1}} are
not distinguishable \wrt\ commutation, and attribute
clarification will remove one of \m{s} and \m{s^{-1}}.
\end{observation}

\begin{lemma}\label{lem:char-comm-f-permutation}
Let \m{s\in\Sym\apply{\CarrierSet}} be a permutation of order
\m{m\in\N} on some set~$\CarrierSet$, and let~$S$ denote the cyclic
permutation group \m{S=\lset{s^j}{0\leq j < m}} generated by~$s$.
For \m{n\in\N}, the group~\m{S} acts on \m{\CarrierSet[n]} by
\m{\apply{\tilde{s},\mathff{x}}\mapsto \tilde{s}\circ\mathff{x}} and
thereby partitions all \nbdd{n}tuples into orbits; let
\m{\lset{\mathff{x}_t}{t\in T}\subs \CarrierSet[n]} be a transversal of
these orbits. The permutation group~$S$ also acts
on~$\sigma\subs\CarrierSet[3]$;
let \m{\lset{\mathff{x}_t}{t\in T'}\subs\sigma} be a transversal of the
corresponding orbits. Moreover, consider any \m{f\in\Ops[n]}.
\begin{enumerate}[(a)]
\item\label{item:perm-f}
      We have \m{s\commuteswith f} if and only if
      \m{f\apply{s^j\circ\mathff{x}_t} =
         s^j\apply{f\apply{\mathff{x}_t}}}
      holds for all \m{0\leq j<m} and all \m{t\in T}.
\item\label{item:perm-maj}
      If \m{n=3} and \m{f\in\Maj}, we have
      \m{s\commuteswith f} if and only if
      \m{f\apply{s^j\circ\mathff{x}_t} =
         s^j\apply{f\apply{\mathff{x}_t}}}
      holds for all \m{0\leq j<m} and all \m{t\in T'}.
\end{enumerate}
\end{lemma}
\begin{proof}
Statement~\eqref{item:perm-maj} follows from~\eqref{item:perm-f} by
noting that for \m{n=3}, the index set~$T'$ can be extended to some
$T\supseteq T'$ such that
\m{\lset{\mathff{x}_t}{t\in T}\sups\lset{\mathff{x}_t}{t\in T'}} forms
a transversal of~\m{\CarrierSet[n]}. Now for \m{f\in\Maj} the part of
the condition in~\eqref{item:perm-f} corresponding to
\m{t\in T\setminus T'} is automatically true as seen in
Observation~\ref{obs:comm-maj-on-sigma}.
\par
Necessity in~\eqref{item:perm-f} is clear, as \m{s\in \cent{\set{f}}}
entails \m{s^\ell\in \cent{\set{f}}} for all \m{\ell\in \N}, and so we have
\m{f\apply{s^\ell\circ \mathff{x}} = s^\ell\apply{f\apply{\mathff{x}}}}
for all \m{\mathff{x}\in\CarrierSet[n]}.
Conversely, if \m{\mathff{x}= s^j\circ\mathff{x}_t} with \m{0\leq j<m}
is a tuple in the orbit represented by~$\mathff{x}_t$, then
\[f\apply{s\circ\mathff{x}}= f\apply{s^{j+1}\circ\mathff{x}_t}
=s^{j+1}\apply{f\apply{\mathff{x}_t}} =
s\apply{s^j\apply{f\apply{\mathff{x}_t}}} =
s\apply{f\apply{s^j\circ\mathff{x}_t}}=s\apply{f\apply{\mathff{x}}},\]
where the second equality is trivial for \m{j+1=m} and follows by
assumption if \m{j+1<m}.
\end{proof}

The arguments of Lemma~\ref{lem:char-comm-f-permutation} together with
simple considerations which forms the tuple
\m{\apply{s^0\apply{f\apply{\mathff{x}_t}},\dotsc,
          s^{m-1}\apply{f\apply{\mathff{x}_t}}}}
may take, depending on the structure of the permutation
\m{s\in\Sym\apply{\CarrierSet}}, are the basis for the following
characterisations. Their proofs will be omitted as the details are
quite straightforward.
\par

Subsequently, we separate the permutations of \m{\set{0,1,2,3}} by their
number of fixed points.
Following the path outlined
in~\cite[3.1.1]{MachidaRosenbergGoldsternCentralisingMonoids}, we shall first describe a
slightly more general case for specific \m{f\in\Ops[3]}, which we shall
subsequently specialise to that of a majority operation.
\par

\subsubsection{General case}
Here we first consider ternary functions \m{f\in\Ops[3]} satisfying the
identities
\[f\apply{x,y,x} \approx f\apply{x,x,z}\approx x.\]
These are common generalisations of majority operations and ternary
semiprojections on the left\dash{}most coordinate.
\par

To begin we deal with permutations \m{s\in\Sym\apply{4}} without any
fixed point. Such a function necessarily is either of the form
\m{s=\apply{abcd}} or \m{s=\apply{ab}\apply{cd}},
wherein \m{\set{a,b,c,d}=\set{0,1,2,3}}.
\begin{lemma}\label{lem:gen-perm-0}
Let \m{s=\apply{a b c d}} with \m{\set{a,b,c,d}=\set{0,1,2,3}} and
suppose \m{f\in\Op[3]{4}} satisfies \m{f\apply{x,y,x}\approx
f\apply{x,x,y}\approx x}. Then \m{f\commuteswith s} if and only if
each of the following quadruples belongs to the set
\m{\set{\apply{a,b,c,d},\apply{b,c,d,a},\apply{c,d,a,b},\apply{d,a,b,c}}}:
\begin{align*}
\apply{f\apply{a,b,b},f\apply{b,c,c},f\apply{c,d,d},f\apply{d,a,a}}, \\
\apply{f\apply{a,c,c},f\apply{b,d,d},f\apply{c,a,a},f\apply{d,b,b}}, \\
\apply{f\apply{a,d,d},f\apply{b,a,a},f\apply{c,b,b},f\apply{d,c,c}}, \\
\apply{f\apply{a,b,c},f\apply{b,c,d},f\apply{c,d,a},f\apply{d,a,b}}, \\
\apply{f\apply{a,b,d},f\apply{b,c,a},f\apply{c,d,b},f\apply{d,a,c}}, \\
\apply{f\apply{a,c,b},f\apply{b,d,c},f\apply{c,a,d},f\apply{d,b,a}}, \\
\apply{f\apply{a,d,b},f\apply{b,a,c},f\apply{c,b,d},f\apply{d,c,a}}, \\
\apply{f\apply{a,c,d},f\apply{b,d,a},f\apply{c,a,b},f\apply{d,b,c}}, \\
\apply{f\apply{a,d,c},f\apply{b,a,d},f\apply{c,b,a},f\apply{d,c,b}}.
\end{align*}
\end{lemma}

\begin{lemma}\label{lem:gen-perm-02}
Let \m{s=\apply{a b}\apply{c d}} with \m{\set{a,b,c,d}=\set{0,1,2,3}} and
suppose \m{f\in\Op[3]{4}} satisfies \m{f\apply{x,y,x}\approx
f\apply{x,x,y}\approx x}. Then \m{f\commuteswith s} if and only if
each of the following sets either equals \m{\set{a,b}} or
\m{\set{c,d}}:
\begin{align*}
\set{f\apply{a,b,b},f\apply{b,a,a}}, &&
  \set{f\apply{a,b,c},f\apply{b,a,d}}, &&
  \set{f\apply{a,c,d},f\apply{b,d,c}}, \\
\set{f\apply{a,c,c},f\apply{b,d,d}}, &&
  \set{f\apply{a,b,d},f\apply{b,a,c}}, &&
  \set{f\apply{a,d,c},f\apply{b,c,d}}, \\
\set{f\apply{a,d,d},f\apply{b,c,c}}, &&
  \set{f\apply{a,c,b},f\apply{b,d,a}}, &&
  \set{f\apply{c,a,d},f\apply{d,b,c}}, \\
\set{f\apply{c,a,a},f\apply{d,b,b}}, &&
  \set{f\apply{a,d,b},f\apply{b,c,a}}, &&
  \set{f\apply{d,a,c},f\apply{c,b,d}}, \\
\set{f\apply{c,b,b},f\apply{d,a,a}}, &&
  \set{f\apply{c,a,b},f\apply{d,b,a}}, &&
  \set{f\apply{c,d,a},f\apply{d,c,b}}, \\
\set{f\apply{c,d,d},f\apply{d,c,c}}, &&
  \set{f\apply{d,a,b},f\apply{c,b,a}}, &&
  \set{f\apply{d,c,a},f\apply{c,d,b}}.
\end{align*}
\end{lemma}

Next we consider permutations \m{s\in\Sym\apply{4}} with exactly one
fixed point. Such a function necessarily is of the form
\m{s=\apply{a b c}\apply{d}} where \m{\set{a,b,c,d}=\set{0,1,2,3}}.
\begin{lemma}\label{lem:gen-perm-1}
Let \m{s=\apply{a b c}\apply{d}} with \m{\set{a,b,c,d}=\set{0,1,2,3}} and
suppose \m{f\in\Op[3]{4}} satisfies \m{f\apply{x,y,x}\approx
f\apply{x,x,y}\approx x}. Then \m{f\commuteswith s} if and only if
each of the following triples belongs to the set
\m{\set{\apply{a,b,c},\apply{b,c,a},\apply{c,a,b},\apply{d,d,d}}}:
\begin{align*}
\apply{f\apply{a,b,b},f\apply{b,c,c},f\apply{c,a,a}}, &&
\apply{f\apply{a,c,c},f\apply{b,a,a},f\apply{c,b,b}}, \\
\apply{f\apply{a,d,d},f\apply{b,d,d},f\apply{c,d,d}}, &&
\apply{f\apply{d,a,a},f\apply{d,b,b},f\apply{d,c,c}}, \\
\apply{f\apply{a,b,c},f\apply{b,c,a},f\apply{c,a,b}}, &&
\apply{f\apply{a,c,b},f\apply{b,a,c},f\apply{c,b,a}}, \\
\apply{f\apply{a,b,d},f\apply{b,c,d},f\apply{c,a,d}}, &&
\apply{f\apply{b,a,d},f\apply{c,b,d},f\apply{a,c,d}}, \\
\apply{f\apply{a,d,b},f\apply{b,d,c},f\apply{c,d,a}}, &&
\apply{f\apply{b,d,a},f\apply{c,d,b},f\apply{a,d,c}}, \\
\apply{f\apply{d,a,b},f\apply{d,b,c},f\apply{d,c,a}}, &&
\apply{f\apply{d,b,a},f\apply{d,c,b},f\apply{d,a,c}}.
\end{align*}
\end{lemma}

Since a permutation \m{s\in\Sym\apply{4}} with at least three fixed points
must be the identity, which is uninteresting to us, the only remaining case
is that of exactly two fixed points. Such a permutation must be of the form
\m{s = \apply{ab}\apply{c}\apply{d}} where \m{\set{a,b,c,d}=\set{0,1,2,3}}.
\begin{lemma}\label{lem:gen-perm-2}
Let \m{s=\apply{a b}\apply{c}\apply{d}} with
\m{\set{a,b,c,d}=\set{0,1,2,3}} and
consider an operation \m{f\in\Op[3]{4}} satisfying
\m{f\apply{x,y,x}\approx f\apply{x,x,y}\approx x}.
Then \m{f\commuteswith s} if and only if
\m{\set{f\apply{c,d,d},f\apply{d,c,c}}\subs\set{c,d}} and each of the
following sets either equals \m{\set{a,b}}, \m{\set{c}} or \m{\set{d}}:
\begin{align*}
\set{f\apply{a,b,b},f\apply{b,a,a}}, &&
  \set{f\apply{a,b,c},f\apply{b,a,c}}, &&
  \set{f\apply{a,c,d},f\apply{b,c,d}}, \\
\set{f\apply{a,c,c},f\apply{b,c,c}}, &&
  \set{f\apply{a,b,d},f\apply{b,a,d}}, &&
  \set{f\apply{a,d,c},f\apply{b,d,c}}, \\
\set{f\apply{a,d,d},f\apply{b,d,d}}, &&
  \set{f\apply{a,c,b},f\apply{b,c,a}}, &&
  \set{f\apply{c,a,d},f\apply{c,b,d}}, \\
\set{f\apply{c,a,a},f\apply{c,b,b}}, &&
  \set{f\apply{a,d,b},f\apply{b,d,a}}, &&
  \set{f\apply{d,a,c},f\apply{d,b,c}}, \\
\set{f\apply{d,a,a},f\apply{d,b,b}}, &&
  \set{f\apply{c,a,b},f\apply{c,b,a}}, &&
  \set{f\apply{c,d,a},f\apply{c,d,b}}, \\
                                     &&
  \set{f\apply{d,a,b},f\apply{d,b,a}}, &&
  \set{f\apply{d,c,a},f\apply{d,c,b}}.
\end{align*}
\end{lemma}

\subsubsection{Majority operations}
Now we specialise the previous characterisations to the case of majority
operations by deleting those parts of the condition that are automatically
fulfilled.
\par

Again we start with permutations having just one cycle.
\begin{lemma}\label{lem:maj-perm-0}
Let \m{s=\apply{a b c d}} with \m{\set{a,b,c,d}=\set{0,1,2,3}} and
suppose \m{f\in\Op[3]{4}} is a majority operation.
Then \m{f\commuteswith s} if and only if
each of the following quadruples belongs to the set
\m{\set{\apply{a,b,c,d},\apply{b,c,d,a},\apply{c,d,a,b},\apply{d,a,b,c}}}:
\begin{align*}
\apply{f\apply{a,b,c},f\apply{b,c,d},f\apply{c,d,a},f\apply{d,a,b}}, \\
\apply{f\apply{a,b,d},f\apply{b,c,a},f\apply{c,d,b},f\apply{d,a,c}}, \\
\apply{f\apply{a,c,b},f\apply{b,d,c},f\apply{c,a,d},f\apply{d,b,a}}, \\
\apply{f\apply{a,d,b},f\apply{b,a,c},f\apply{c,b,d},f\apply{d,c,a}}, \\
\apply{f\apply{a,c,d},f\apply{b,d,a},f\apply{c,a,b},f\apply{d,b,c}}, \\
\apply{f\apply{a,d,c},f\apply{b,a,d},f\apply{c,b,a},f\apply{d,c,b}}.
\end{align*}
\end{lemma}

The function value in the first entry of each quadruple can be any of
the four elements in~\m{\set{0,1,2,3}}, but once it is chosen, all the
other values in the quadruple are uniquely determined. Hence, we have
only \m{4^6=2^{12}} majority functions to consider.
\par

From Lemma~\ref{lem:maj-perm-0} we obtain three concrete
conditions, corresponding to the permutations
\m{\apply{0123}}, 
\m{\apply{0132}}  
and
\m{\apply{0213}}, 
and their inverses, respectively.
\par

\begin{corollary}\label{cor:maj-perm-0}%
For a majority operation \m{f\in\Op[3]{4}} we have
\begin{itemize}
\item \condcyclic{0}{1}{2}{3}{C1}
\item \condcyclic{0}{1}{3}{2}{C2}
\item \condcyclic{0}{2}{1}{3}{C3}
\end{itemize}
\end{corollary}

\begin{lemma}\label{lem:maj-perm-02}
Let \m{s=\apply{a b}\apply{c d}} with \m{\set{a,b,c,d}=\set{0,1,2,3}} and
suppose \m{f\in\Op[3]{4}} is a majority operation. Then \m{f\commuteswith s}
if and only if each of the following sets either equals \m{\set{a,b}} or
\m{\set{c,d}}:
\begin{align*}
  \set{f\apply{a,b,c},f\apply{b,a,d}}, &&
  \set{f\apply{a,c,d},f\apply{b,d,c}}, \\
  \set{f\apply{a,b,d},f\apply{b,a,c}}, &&
  \set{f\apply{a,d,c},f\apply{b,c,d}}, \\
  \set{f\apply{a,c,b},f\apply{b,d,a}}, &&
  \set{f\apply{c,a,d},f\apply{d,b,c}}, \\
  \set{f\apply{a,d,b},f\apply{b,c,a}}, &&
  \set{f\apply{d,a,c},f\apply{c,b,d}}, \\
  \set{f\apply{c,a,b},f\apply{d,b,a}}, &&
  \set{f\apply{c,d,a},f\apply{d,c,b}}, \\
  \set{f\apply{d,a,b},f\apply{c,b,a}}, &&
  \set{f\apply{d,c,a},f\apply{c,d,b}}.
\end{align*}
\end{lemma}

Here, the first listed element of each pair may take any of the four
values in~\m{\set{0,1,2,3}}, while the second listed one is uniquely
determined by this choice. We therefore need to work with
\m{4^{12}=2^{24}} majority operations.
Lemma~\ref{lem:maj-perm-02} yields three concrete conditions,
corresponding to the permutations
\m{\apply{01}\apply{23}},
\m{\apply{02}\apply{13}}
and
\m{\apply{03}\apply{12}}.
\par

\begin{corollary}\label{cor:maj-perm-02}
For a majority operation \m{f\in\Op[3]{4}} we have
\begin{itemize}
\item \condtwocycles{0}{1}{2}{3}{D1}
\item \condtwocycles{0}{2}{1}{3}{D2}
\item \condtwocycles{0}{3}{1}{2}{D3}
\end{itemize}
\end{corollary}

Also for permutations with exactly one fixed point we obtain a few simplifications.
\begin{lemma}\label{lem:maj-perm-1}
Let \m{s=\apply{a b c}\apply{d}} with \m{\set{a,b,c,d}=\set{0,1,2,3}} and
suppose \m{f\in\Op[3]{4}} is a majority operation.
Then \m{f\commuteswith s} if and only if
each of the following triples belongs to the set
\m{\set{\apply{a,b,c},\apply{b,c,a},\apply{c,a,b},\apply{d,d,d}}}:
\begin{align*}
\apply{f\apply{a,b,c},f\apply{b,c,a},f\apply{c,a,b}}, &&
\apply{f\apply{a,c,b},f\apply{b,a,c},f\apply{c,b,a}}, \\
\apply{f\apply{a,b,d},f\apply{b,c,d},f\apply{c,a,d}}, &&
\apply{f\apply{b,a,d},f\apply{c,b,d},f\apply{a,c,d}}, \\
\apply{f\apply{a,d,b},f\apply{b,d,c},f\apply{c,d,a}}, &&
\apply{f\apply{b,d,a},f\apply{c,d,b},f\apply{a,d,c}}, \\
\apply{f\apply{d,a,b},f\apply{d,b,c},f\apply{d,c,a}}, &&
\apply{f\apply{d,b,a},f\apply{d,c,b},f\apply{d,a,c}}.
\end{align*}
\end{lemma}

Again, the first function value for any of the eight triples can be
chosen arbitrarily in~\m{\set{0,1,2,3}}, but the remaining two are then
fixed. So we here have to deal with just \m{4^8=2^{16}} majority
functions per unary map.
\par
This time we obtain four concrete conditions, corresponding to the
permutations
\m{\apply{012}\apply{3}},
\m{\apply{013}\apply{2}},
\m{\apply{023}\apply{1}}
and
\m{\apply{123}\apply{0}},
and their respective inverses.
\par

\begin{corollary}\label{cor:maj-perm-1}
For a majority operation \m{f\in\Op[3]{4}} we have
\begin{itemize}
\item \condtriple{0}{1}{2}{3}{E1}
\item \condtriple{0}{1}{3}{2}{E2}
\item \condtriple{0}{2}{3}{1}{E3}
\item \condtriple{1}{2}{3}{0}{E4}
\end{itemize}
\end{corollary}

Finally, we consider again those permutations with precisely two fixed
points.
\begin{lemma}\label{lem:maj-perm-2}
Let \m{s=\apply{a b}\apply{c}\apply{d}} with
\m{\set{a,b,c,d}=\set{0,1,2,3}} and
consider a majority operation \m{f\in\Op[3]{4}}.
Then \m{f\commuteswith s} if and only if each of the
following sets either equals \m{\set{a,b}}, \m{\set{c}} or \m{\set{d}}:
\begin{align*}
  \set{f\apply{a,b,c},f\apply{b,a,c}}, &&
  \set{f\apply{a,c,d},f\apply{b,c,d}}, \\
  \set{f\apply{a,b,d},f\apply{b,a,d}}, &&
  \set{f\apply{a,d,c},f\apply{b,d,c}}, \\
  \set{f\apply{a,c,b},f\apply{b,c,a}}, &&
  \set{f\apply{c,a,d},f\apply{c,b,d}}, \\
  \set{f\apply{a,d,b},f\apply{b,d,a}}, &&
  \set{f\apply{d,a,c},f\apply{d,b,c}}, \\
  \set{f\apply{c,a,b},f\apply{c,b,a}}, &&
  \set{f\apply{c,d,a},f\apply{c,d,b}}, \\
  \set{f\apply{d,a,b},f\apply{d,b,a}}, &&
  \set{f\apply{d,c,a},f\apply{d,c,b}}.
\end{align*}
\end{lemma}

As before, the first member of each unordered pair can be any of the
four elements in the base set, while the second member is then
uniquely determined. This gives \m{4^{12}=2^{24}} choices for majority
operations per unary function.
\par
Lemma~\ref{lem:maj-perm-2} leads us to six concrete conditions, corresponding to the
permutations
\m{\apply{01}\apply{2}\apply{3}},
\m{\apply{02}\apply{1}\apply{3}},
\m{\apply{03}\apply{1}\apply{2}},
\m{\apply{12}\apply{0}\apply{3}},
\m{\apply{13}\apply{0}\apply{2}}
and
\m{\apply{23}\apply{0}\apply{1}}.
\par

\begin{corollary}\label{cor:maj-perm-2}
For a majority operation \m{f\in\Op[3]{4}} we have
\begin{itemize}
\item \condtwofixedpoints{0}{1}{2}{3}{F1}
\item \condtwofixedpoints{0}{2}{1}{3}{F2}
\item \condtwofixedpoints{0}{3}{1}{2}{F3}
\item \condtwofixedpoints{1}{2}{0}{3}{F4}
\item \condtwofixedpoints{1}{3}{0}{2}{F5}
\item \condtwofixedpoints{2}{3}{0}{1}{F6}
\end{itemize}
\end{corollary}

From this section we hence obtain altogether sixteen conditions involving
permutations:
C1, C2, C3, D1, D2, D3, E1, E2, E3, E4, and
F1, F2, F3, F4, F5, F6.

\subsection{Remarks on clarification}\label{subsect:clarification}
While developing the characterisations in this section we have already
briefly touched upon possible attribute clarification. Moreover, we
have commented on upper bounds for the number of majority operations
commuting with a unary function~$s$ of a particular type. By way of two
concrete examples we shall now explain in a bit more detail how the
previous results can be used for attribute clarification, and for
object clarification, as well. The latter is related to the following
task discussed at the end of Section~\ref{sect:method} that awaits
to be performed for every $s\in\Op[1]{4}$ in Section~\ref{sect:results}:
for every \m{f\in \set{s}' = \cent{\set{s}}\cap \Maj[4]} we have to
make a record of \m{\set{f}' = \Fn[1]{\cent{\set{f}}{}}}, and, if
$f_1,f_2\in\set{s}'$ give the same monoid $\set{f_1}'=\set{f_2}'$, then
it is not necessary to store it twice. This is exactly an instance of
object clarification for
\m{\K = \apply{\Maj[4],\Op[1]{4},{\commuteswith}}}, and it is crucial
to make our task of finding all centralising monoids belonging to~$\K$
practicable.

\begin{example}\label{ex:A1}
This example deals with condition~A1 characterising which majority
functions commute with a unary map~$s$ that is constant
on~$\set{1,2,3}$ and attains a distinct second value at~$0$. These
are $16-4=12$ concrete unary functions~$s$, having value tuples
$(0,1,1,1)$,
$(0,2,2,2)$,
$(0,3,3,3)$,
$(1,0,0,0)$,
$(1,2,2,2)$,
\dots,
$(3,1,1,1)$,
$(3,2,2,2)$
on $(0,1,2,3)$. Condition~A1 says that all of them commute with exactly
the same majority operations, namely that
$\cent{\set{s}}\cap\Maj[4]$ equals
\m{\lset{f\in\Maj[4]}{f\fapply{\sigma}\subs\set{1,2,3}}}. Therefore,
these twelve unary maps are identified by attribute clarification.
Comparing the condition~A1 with any of the other characterisations in
Section~\ref{sect:attribute-clarification} shows that no other unary
function commutes with exactly the same majority operations since
these characterisations all lead to conditions A2, \dots, F6 distinct
from~A1. Consequently, the equivalence class of the function
$s=(0,1,1,1)$ under attribute clarification contains exactly the twelve
mentioned functions, and they are henceforth represented by the `name'
A1. With respect to commutation with majority operations, they cannot
be distinguished from each other, and any centralising monoid with
majority witnesses that has attribute~A1 contains, in fact, these
twelve unary maps.
\par
Turning to the task of enumerating
\m{\set{s}'=\cent{\set{s}}\cap\Maj[4]}, we view every $f\in\set{s}'$ as
a tuple of values corresponding to the 24~triples in~$\sigma$, i.e., we
just consider \m{\Restr{f}{\sigma}} instead of~$f$.
Condition~A1 expresses that in this sense $\set{s}'$ can be identified
with $\set{1,2,3}^{24}$ (more exactly with $\set{1,2,3}^\sigma$).
One can now easily iterate over all these \nbdd{24}tuples
$\Restr{f}{\sigma}\in\set{1,2,3}^\sigma$ and compute
$\Fn[1]{\cent{\set{f}}{}}$, only storing it if a previously unseen one emerges.
In this way an object reduced subcontext of
\m{\Bigl(\set{s}',\Op[1]{4},{\commuteswith}\cap
\Bigl(\set{s}'\times\Op[1]{4}\Bigr)\Bigr)} is obtained.
\end{example}

\begin{example}\label{ex:D1}
Here we have a look at condition~D1 describing the majority operations
commuting with the permutation $s=(01)(23)$.
From Observation~\ref{obs:inverse-perm} we know that always~$s$ and its
inverse commute with the same majority operations, and comparing all
the conditions in Section~\ref{sect:attribute-clarification}, we see
that the attribute clarification equivalence class of a
non\dash{}identical permutation is
not bigger than \m{\set{s,s^{-1}}}. Since the D- and F-conditions
correspond to involutions, the 16 conditions for permutations represent
exactly all 23 non\dash{}identical permutations.
\par
As explained after Lemma~\ref{lem:maj-perm-02}, in order to enumerate
$\Restr{f}{\sigma}$ for $f\in\set{s}'$, one loops over all
$\Restr{f}{\sigma'}\in\set{0,1,2,3}^{\sigma'}$, where
\m{\sigma'} is the twelve\dash{}element subset
\begin{multline*}
\{(0,1,2),(0,1,3),(0,2,1),(0,3,1),(0,2,3),(0,3,2),(2,0,3),\\
  (3,0,2),(2,0,1),(2,3,0),(3,0,1),(3,2,0)\}\subs\sigma,
\end{multline*}
and each time one reconstructs $\Restr{f}{\sigma}$ from $\Restr{f}{\sigma'}$ as
by Lemma~\ref{lem:maj-perm-02} or condition~D1. For example, if
$f(0,1,2)=2$, then condition~D1 enforces that $f(1,0,3)=3$ etc.
From~$\Restr{f}{\sigma}$ one derives $\Fn[1]{\cent{\set{f}}{}}$ and discards it if it
already appears in the list of previously computed centralising monoids.
Thereby an object clarified context
\m{\K_s = \Bigl(F_s,\Op[1]{4},{\commuteswith}\cap
\Bigl(F_s\times\Op[1]{4}\Bigr)\Bigr)} with a suitable
\m{F_s\subs \set{s}'} is produced.
\end{example}

\section{Results}\label{sect:results}
Our starting point was the Galois correspondence of commutation between
majority operations and unary operations on
\m{\CarrierSet=\set{0,1,2,3}}, given by the formal context
\m{\K = \apply{\Maj,\Ops[1],{\commuteswith}}}; the item of
interest was the system of Galois closed sets on the side of unary
operations, i.e., the closure system of \emph{intents} in the language
of formal concept analysis.
For each non\dash{}trivial unary map
\m{s\in\Ops[1]\setminus\set{\id_{\CarrierSet},c_0,c_1,c_2,c_3}}
on~$\CarrierSet$, we enumerated \m{\set{s}' = \Maj\cap\cent{\set{s}}}
and obtained an (object) clarified (sub)list of closures
from \m{\apply{\set{f}'}_{f\in\set{s}'}} where
\m{\set{f}' = \Ops[1]\cap\cent{\set{f}}}, cf.\
Subsection~\ref{subsect:clarification}.
These lists correspond to
subcontexts \m{\K_s = \apply{F_s,\Ops[1],{\commuteswith}\cap
\apply{F_s\times\Ops[1]}}} where \m{F_s\subs \set{s}'} is a selection of
majority operations commuting with~$s$.
The object clarification can
of course be performed on the fly while enumerating the closures by
just checking whether the currently produced closure \m{\set{f}'} is
present among the already stored ones. Moreover, we did not actually
work with the full set of attributes~$\Ops[1]$, but used only a
\nbdd{167}element subset \m{S\subs\Ops[1]}, corresponding to the
conditions A1--A7 from Subsection~\ref{subsect:two-element-image}, to the
16 conditions on permutations C1--C3, D1--D3, E1--E4, F1--F6 from
Subsection~\ref{subsect:permutation}, and to the $144$~operations \m{s=u_n}
with a \nbdd{3}element image where \m{n=\sum_{j=0}^3 s(j)4^{3-j}}. That
is, we worked with
\m{\K_s'=\apply{F_s,S,{\commuteswith}\cap \apply{F_s\times S}}} instead
of~$\K_s$. To avoid too much computational overhead we did not object
reduce these subcontexts~$\K_s'$ while producing them. In the next
step we joined all these lists into one formal context
\m{\K'=\apply{F^\dagger,S,{\commuteswith}\cap\apply{F^\dagger\times S}}},
having as set of objects \m{F^\dagger
=\set{f_0}\mathbin{\dot{\cup}}\mathop{\dot{\bigcup}}_{s\in\Ops[1]\setminus\set{\id_{\CarrierSet},c_0,c_1,c_2,c_3}} F_s}
(including one majority operation~$f_0$ that only commutes with
constants and the identity map).
The context~$\K'$ contains \m{8\,119} majority functions as objects and
then had to be object clarified again. In a final step we object
reduced the resulting context (removing in particular~$f_0$) to get a
context
\m{\K''=\apply{\tilde{F},S,{\commuteswith}\cap\apply{\tilde{F}\times
S}}} with
\m{\tilde{F}\subs\bigcup_{s\in
\Ops[1]\setminus\set{\id_{\CarrierSet}, c_0,c_1,c_2,c_3}}F_s}, having
only $392$~majority functions as objects.
The context \m{\K''} is presented in
Tables~\ref{table:obj-reduced-clarified-context-1}
to~\ref{table:obj-reduced-clarified-context-4} below,
its set~$\tilde{F}$ of objects is given in Table~\ref{tbl-obj-K2}.
The correctness of~\m{\K''} can be checked with basic programming skills; that
it is sufficient to completely represent the original context~$\K$ is
the content of the subsequent theorem and depends on the correct
execution of the calculations described previously. We note that
from~$\K''$ one can easily obtain a fully clarified and reduced context
\m{\K'''=\apply{\tilde{F},\tilde{T},{\commuteswith}\cap\apply{\tilde{F}\times\tilde{T}}}} that has an
isomorphic lattice of closures, but takes up less space:
\m{\crd{\tilde{F}}=392}, \m{\crd{\tilde{T}}=155}.

\begin{theorem}\label{thm:context-for-all-centralising-monoids}
The formal contexts~$\K$ and
$\apply{F^\dagger,\Ops[1],{\commuteswith}\cap\bigl(F^\dagger\times\Ops[1]\bigr)}$
have the same closure system of intents,
and, ordered under inclusion, this lattice of intents (i.e., of all
centralising monoids with majority witnesses over
\m{\CarrierSet=\set{0,1,2,3}}) is isomorphic to
the one of\/~$\K'$, $\K''$ and~\m{\K'''}.
\end{theorem}
\begin{proof}
It is well known and easy to see that the operations of clarifying and
reduction (for objects and attributes) do not change the lattice
structure of the resulting closure system. Moreover, by some
bookkeeping (adding reducible elements to a closed set if all members
of the representing intersection are present, and adding elements that
are identified by clarification), one can reconstruct both closure
systems from the ones of a clarified (and/or reduced) context.
\par
Thus, it remains to argue that
\m{\apply{F^\dagger,\Ops[1],{\commuteswith}\cap(F^\dagger\times\Ops[1])}}
has the same set of intents as~$\K$. This closure system consists
of~$F'$ for any \m{F\subs\Maj} where \m{F' = \bigcap_{f\in F}\set{f}'}.
Only two cases can occur: either there is some \m{f\in F} such that
\m{\set{f}' = \set{\id_{\CarrierSet},c_0,c_1,c_2,c_3}}, in which case
\m{F' = \set{\id_{\CarrierSet},c_0,c_1,c_2,c_3}=\set{f_0}'}, or~\m{F'}
can be obtained as an intersection of sets \m{\set{g}'} with \m{g\in
F^\dagger}, i.e., \m{F'=G'} for some \m{G\subs F^\dagger}. This
holds because, if there is no \m{f\in F} as above, then for every
\m{f\in F} we have \m{\set{f}' \supsetneq
\set{\id_{\CarrierSet},c_0,\dots,c_3}}, i.e., there is some
$s\in\cent{\set{f}} \cap
\apply{\Ops[1]\setminus\set{\id_{\CarrierSet},c_0,\dots,c_3}}$, and
this unary map then satisfies \m{f\in \Maj\cap\cent{\set{s}} =
\set{s}'}. Hence, \m{\set{f}'} is part of the list
\m{\apply{\set{g}'}_{g\in\set{s}'}}, and thus
\m{\set{f}'=\set{g}'} for some \m{g\in F_s}. For~$f$ was arbitrary
in~$F$, we obtain \m{F'=\bigcap_{f\in F}\set{f}' = \bigcap_{g\in
G}\set{g}'} where \m{G\subs\bigcup_{s\in
\Ops[1]\setminus\set{\id_{\CarrierSet},c_0,c_1,c_2,c_3}} F_s\subs
F^{\dagger}}.
\end{proof}

It is quite remarkable that only~$392$ out of all~$4^{24}$ majority operations are
needed to distinguish all centralising monoids with majority witnesses
on a four\dash{}element set.

\begin{corollary}\label{cor:number-of-centralising-monoids}
The number of centralising monoids on~$\CarrierSet=\set{0,1,2,3}$ with
majority witnesses is~$1\,715$.
\end{corollary}
\begin{proof}
One may run any standard algorithm for enumerating all closure systems
of a Galois correspondence on~\m{\K''}, given by
Tables~\ref{table:obj-reduced-clarified-context-1}
to~\ref{table:obj-reduced-clarified-context-4}.
The result is that there are (at least) $1\,715$ centralising monoids
given by majority witnesses on four\dash{}element domains.
That there are not more monoids of this type follows from
Theorem~\ref{thm:context-for-all-centralising-monoids}.
\end{proof}

\begin{remark}\label{rem:algorithms-and-implementations}
The context~$\K''$ is small enough that no specialised
enumeration algorithms for the closure system of intents are needed;
even Ganter's fundamental \textsc{Next Closure} algorithm
(see~\cite[Chapter~2.4, p.~44 et
seqq.]{GanterObiedkovConceptualExploration}) can be successfully
applied. There are ready to use implementations available for
download that allow anybody to check that there are at least $1\,715$
centralising monoids with majority witnesses (those produced
from~$\K''$) on a four\dash{}element domain. See, e.g., Uta Priss's
homepage~\cite{PrissFCA} for a list of software. Among those we
wish to highlight the implementations \textsc{ConExp} by
S.~Yevtushenko, and the more recent \textsc{Conexp-clj} by D.~Borchmann
and co-authors that we found particularly useful.
\end{remark}

We add a consequence that is relevant for determining all
maximal centralising monoids on four\dash{}element carrier sets. To
organise the presentation, we shall use the following fact about
conjugation with \emph{inner automorphisms}.
\begin{observation}\label{obs:conjugation}
For $m,n\in\N$, functions $f,\tilde{f}\in\Op[n]{\CarrierSet}$,
$g,\tilde{g}\in\Op[m]{\CarrierSet}$ and \m{s\in\Ops[1]} we have
that \m{s\colon
\algwops{\CarrierSet}{f,g}\to\algwops{\CarrierSet}{\tilde{f},\tilde{g}}}
is an isomorphism if and only if \m{s\in\Sym(\CarrierSet)} and
$\tilde{f}(\mathff{x}) = f^{s}(\mathff{x})
                         \defeq s(f(s^{-1}\circ\mathff{x}))$
for all $\mathff{x}\in\CarrierSet[n]$
and $\tilde{g}(\mathff{y}) = g^{s}(\mathff{y})
                             \defeq s(g(s^{-1}\circ\mathff{y}))$
for all \m{\mathff{y}\in\CarrierSet[m]}.
\par
Therefore, for any \m{s\in\Sym(\CarrierSet)}, we have
$f\commuteswith g$ if and only if $f^{s}\commuteswith g^{s}$,
which means that
$\set{f^{s}}^{*(m)} = \lset{g^{s}}{g\in\set{f}^{*(m)}}$.
\end{observation}

According to Observation~\ref{obs:conjugation}, the centralising
monoid of the conjugate \m{f^{s}} of a majority operation \m{f\in\Maj}
is uniquely determined by conjugating every unary function in
\m{\Fn[1]{\cent{\set{f}}{}}} by~\m{s\in\Sym(\CarrierSet)}. The respective
monoids are then isomorphic by sending any \m{u\in\Fn[1]{\cent{\set{f}}{}}} to
its conjugate \m{u^{s}\in\Fn[1]{\cent{\set{f^s}}{}}}.
\par

In this way we shall see that maximal centralising monoids
on~$\set{0,1,2,3}$ with majority witnesses have witnesses from
19~distinct conjugacy classes. Along the same lines we could classify
all the 392~majority functions from Table~\ref{tbl-obj-K2} into
90~conjugacy classes but we have suppressed such a list for brevity.

\begin{corollary}\label{cor:maximal-centralising-monoids}
Among the~$1\,715$ centralising monoids on~$\CarrierSet=\set{0,1,2,3}$
with majority witnesses there are~$147$ that are maximal with respect
to set inclusion. They are given as Galois derivatives
$\Fn[1]{\cent{\set{f_n}}{}}$ of those majority functions~$f_n$
(cp.~Table~\ref{tbl-obj-K2}) where
$n\in\{1\cb 5\cb 8\cb 9\cb 14\cb 18\cb 20\cb 46\cb 50\cb 73\cb 82\cb
92\cb 94\cb 101\cb 104\cb 184\cb 237\cb 274\cb 283\}$,
or they are isomorphic to the previous~19 by conjugation with an inner
automorphism from~$\Sym(4)$, giving the remaining
$n\in\{2\cb 3\cb 4\cb 16\cb 19\cb 23\cb 24\cb 25\cb 28\cb 29\cb 30\cb
32\cb 33\cb 36\cb 38\cb 40\cb 48\cb 52\cb 53\cb 56\cb 58\cb 64\cb 66\cb
67\cb 70\cb 72\cb 75\cb 77\cb 79\cb 85\cb 98\cb 100\cb 102\cb 106\cb
108\cb 110\cb 118\cb 119\cb 121\cb 125\cb 127\cb 128\cb 129\cb 132\cb
134\cb 137\cb 143\cb 145\cb 146\cb 148\cb 152\cb 154\cb 156\cb 158\cb
159\cb 161\cb 162\cb 165\cb 166\cb 170\cb 177\cb 181\cb 182\cb 185\cb
187\cb 189\cb 190\cb 194\cb 201\cb 204\cb 207\cb 208\cb 210\cb 211\cb
212\cb 215\cb 216\cb 218\cb 227\cb 228\cb 229\cb 231\cb 232\cb 242\cb
243\cb 244\cb 245\cb 246\cb 247\cb 248\cb 249\cb 254\cb 255\cb 256\cb
257\cb 258\cb 259\cb 260\cb 262\cb 263\cb 267\cb 268\cb 269\cb 270\cb
271\cb 272\cb 276\cb 284\cb 286\cb 288\cb 291\cb 296\cb 298\cb 302\cb
303\cb 308\cb 310\cb 311\cb 315\cb 318\cb 322\cb 324\cb 327\cb 330\cb
335\cb 336\cb 339\cb 342\}$,
see Tables~\ref{tbl:maximal-monoids-representatives},
\ref{tbl:maximal-monoids-conjugates},
\ref{tbl:maximal-monoids-1} and~\ref{tbl:maximal-monoids-2}.
\end{corollary}
\begin{proof}
Run through the list of all~$1\,714$ centralising monoids on~$A$ with
majority witnesses except for~$\Ops[1]$ and check whether there is
another one of these that is strictly larger than the currently
considered monoid~$M$. If yes, ignore~$M$, otherwise add~$M$ to the
list of maximal ones. Alternatively, one may only iterate over
all~$392$ majority functions featuring as objects in~$\K''$, since
(similarly as for maximal clones) every maximal centralising monoid
must be given as the Galois derivative~\m{\Fn[1]{\cent{\set{f}}{}}} of
a single non\dash{}trivial element on the opposite side of the Galois
connection.
\par
To obtain the conjugacy classes, one picks the first majority function
from the list and removes all its conjugates from the list, putting
them into a container of their own. One continues like this until the
list is empty.
\end{proof}

We conclude with some remarks on semiprojections.
As for majority operations, one can observe that for any unary
function \m{s\in\Ops[1]} and all \m{\mathff{x}=\apply{x_j}_{j\in n}\in \CarrierSet[n]} with
repetitions, these repetitions persist in~\m{s\circ \mathff{x}}, and so
\m{s(f(\mathff{x})) = s(x_i) = f(s\circ\mathff{x})} for any \nbdd{n}ary
semiprojection~$f$ on the place~$i$.
This shows the following analogue of Observation~\ref{obs:comm-maj-on-sigma}:
\begin{observation}\label{obs:comm-sproj-on-sigma}
Let \m{n\in \N} be at least two, \m{i\in n}, \m{f\in\Ops[n]} a
semiprojection on the position indexed by~$i$ and \m{s\in\Ops[1]} a
unary operation on some set~$\CarrierSet$.
Then we have \m{s\commuteswith f} if and only if
\m{s(f(\mathff{x}))=f(s\circ\mathff{x})} holds for all injective tuples
\m{\mathff{x}\in \CarrierSet[n]} (i.e., those without repetitions).
\end{observation}

For \m{n=3} a tuple \m{\mathff{x}\in\CarrierSet[3]} with repetitions
automatically is a majority triple \m{\mathff{x}\in
\CarrierSet[3]\setminus\sigma} and vice versa. By their intrinsic
nature, on these triples, the values of majority operations and ternary
semiprojections are completely preordained. So both types of
functions are uniquely determined by their restrictions to~\m{\sigma},
and the commutation property is intimately connected to the triples
in~$\sigma$. This may lead one to think that knowing the restrictions
to~$\sigma$ is sufficient to decide commutation with unary operations.
The latter is true when studying permutations in place of arbitrary
unary operations, but not in general, as the following example shows.

\begin{example}\label{ex:tern-sproj-not-representable-by-maj}
Consider \m{s\defeq u_{26}\in\Op[1]{4}}, that is,
\m{s\circ\apply{0,1,2,3}=\apply{0,1,2,2}} and some
\m{f\in\Maj[4]\cap\cent{\set{s}}}, e.g., \m{f=f_3} from
Table~\ref{table:obj-reduced-clarified-context-1}.
Furthermore, let \m{g\in\Op[3]{4}} be the unique ternary semiprojection
on the left\dash{}most coordinate given by
\m{\Restr{f}{\sigma}=\Restr{g}{\sigma}}.
Now we have
$s(g(1,2,3))=s(f(1,2,3))=f(s\circ(1,2,3))=f(1,2,2)=2$,
while
\m{g(s\circ(1,2,3)) = g(1,2,2)=1}, so
\m{s\not\commuteswith g}. Hence,
\m{s\in\Fn[1]{\cent{\set{f}}{}}\setminus\cent{\set{g}}} demonstrates
that \m{\Fn[1]{\cent{\set{f}}{}}\neq \Fn[1]{\cent{\set{g}}{}}} although
\m{\Restr{f}{\sigma}=\Restr{g}{\sigma}}. This means that, with respect
to unary centralisers, ternary semiprojections cannot be replaced in a
straightforward way by majority operations (as is possible in the
context of centralising groups,
see~\cite[Proposition~16(c) and Corollary~17, p.~295]{BehPoeschelCentralisingGroups}).
\par
As the example shows, the problem is that agreement on~$\sigma$ does
not ensure that~$f$ and~$g$ coincide on all tuples
\m{s\circ\mathff{x}} where \m{\mathff{x}\in\sigma}, for those may lie
outside~$\sigma$. This is never an issue when~$s$ is a permutation
(as confirmed by the results in Subsection~\ref{subsect:permutation}).
\end{example}

Accordingly, for the task of finding all maximal centralising monoids
on a four\dash{}element set, there remain only two parts. The major
piece is given by considering monoids with ternary or quaternary semiprojection
witnesses, which should be approachable by similar techniques as the
ones presented in this article. The second, and computationally less
costly, hence minor, part will be to deal with centralising monoids
given by the remaining witnesses of smaller arity described in
Rosenberg's Theorem~\ref{thm:Rosenberg-minimal-clones}, i.e., very
specific unary operations, binary idempotent functions and special
ternary minority functions.

\section*{Acknowledgements}
The author would like to express his warmly felt gratitude to Hajime
Machida for several inspiring discussions on centralising monoids and
centraliser clones, and for his persistence at giving many
insightful and encouraging remarks.
\par
\enlargethispage{2\baselineskip}
He would also like to thank Zarathustra Brady for introducing him to
arXiv's ancillary files feature.

\input{centralisingmonoids4maj_ref.tex}

\smallskip

\myContact{\CorrespondingAuthor}{%
\TUWname,
\InstitutDMG,
\AddressDMG,
\PostleitzahlWien\\
\url{https://orcid.org/0000-0003-0050-8085}}{behrisch@logic.at}

\begin{table}
\includegraphics[width=\linewidth]{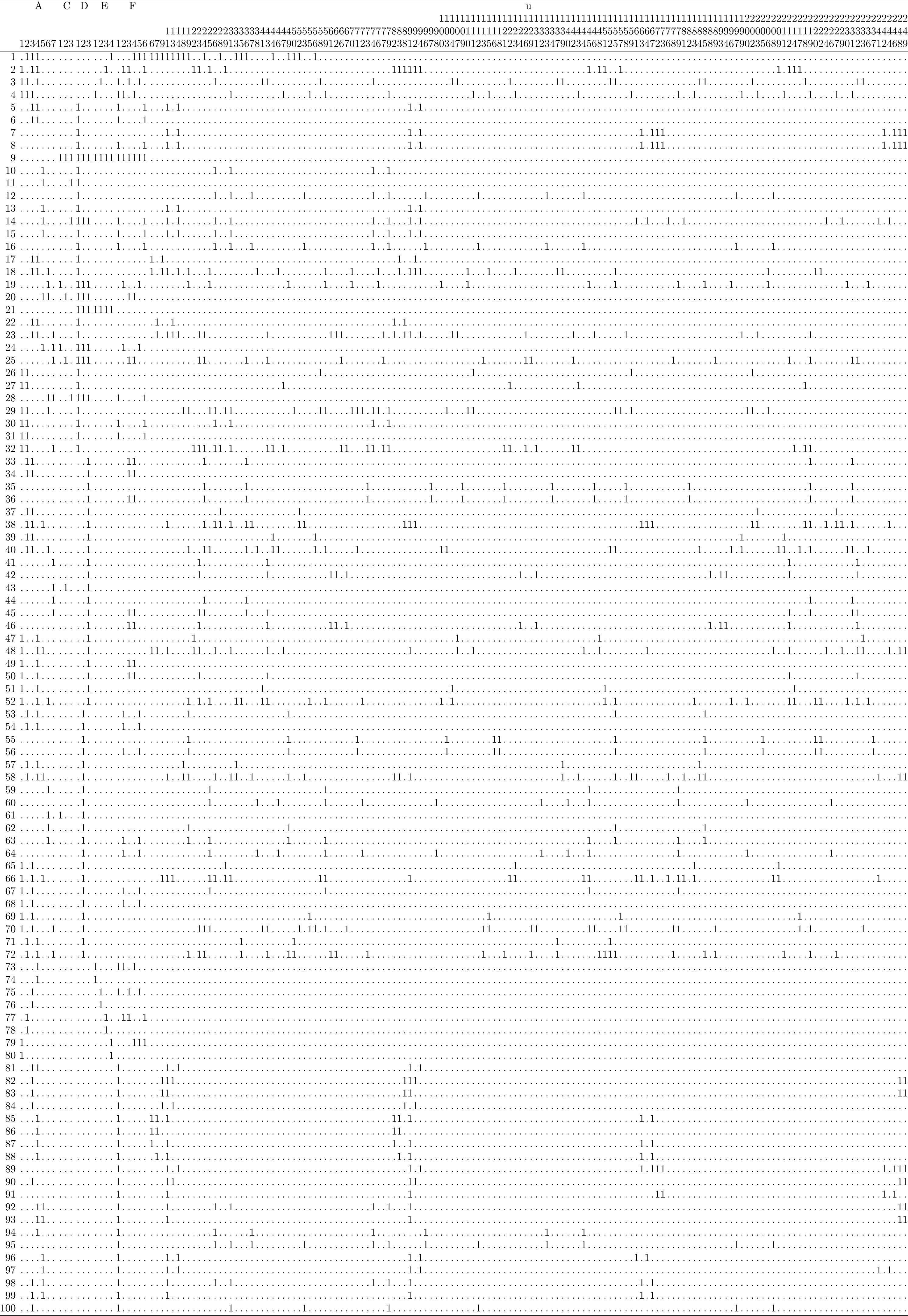}%
\caption{The context $\K''$, objects 1--100 (for attributes, see before
Theorem~\ref{thm:context-for-all-centralising-monoids}).}
\label{table:obj-reduced-clarified-context-1}
\end{table}
\begin{table}
\includegraphics[width=\linewidth]{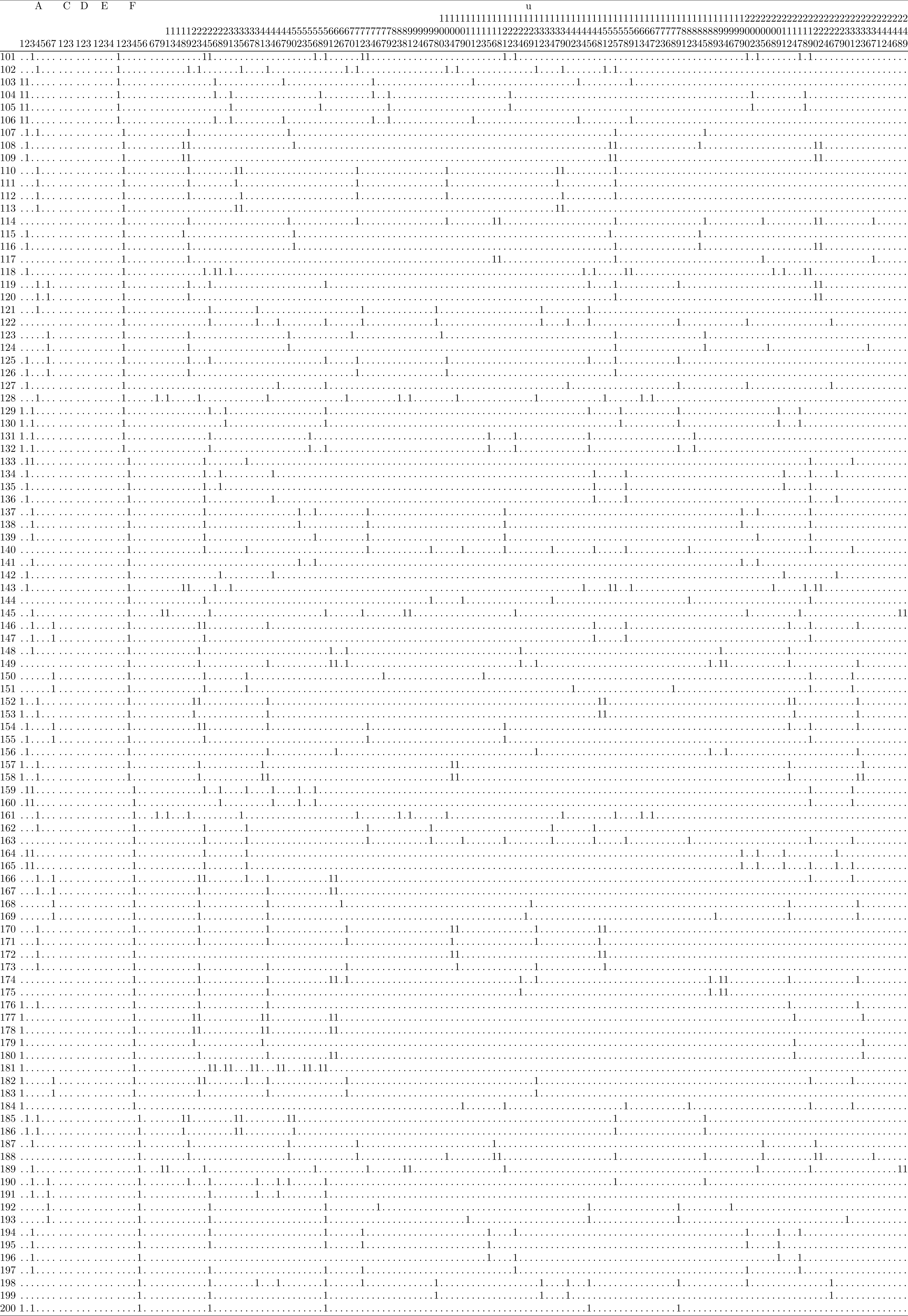}%
\caption{The context $\K''$, objects 101--200 (for attributes, see before Theorem~\ref{thm:context-for-all-centralising-monoids}).}
\label{table:obj-reduced-clarified-context-2}
\end{table}
\begin{table}
\includegraphics[width=\linewidth]{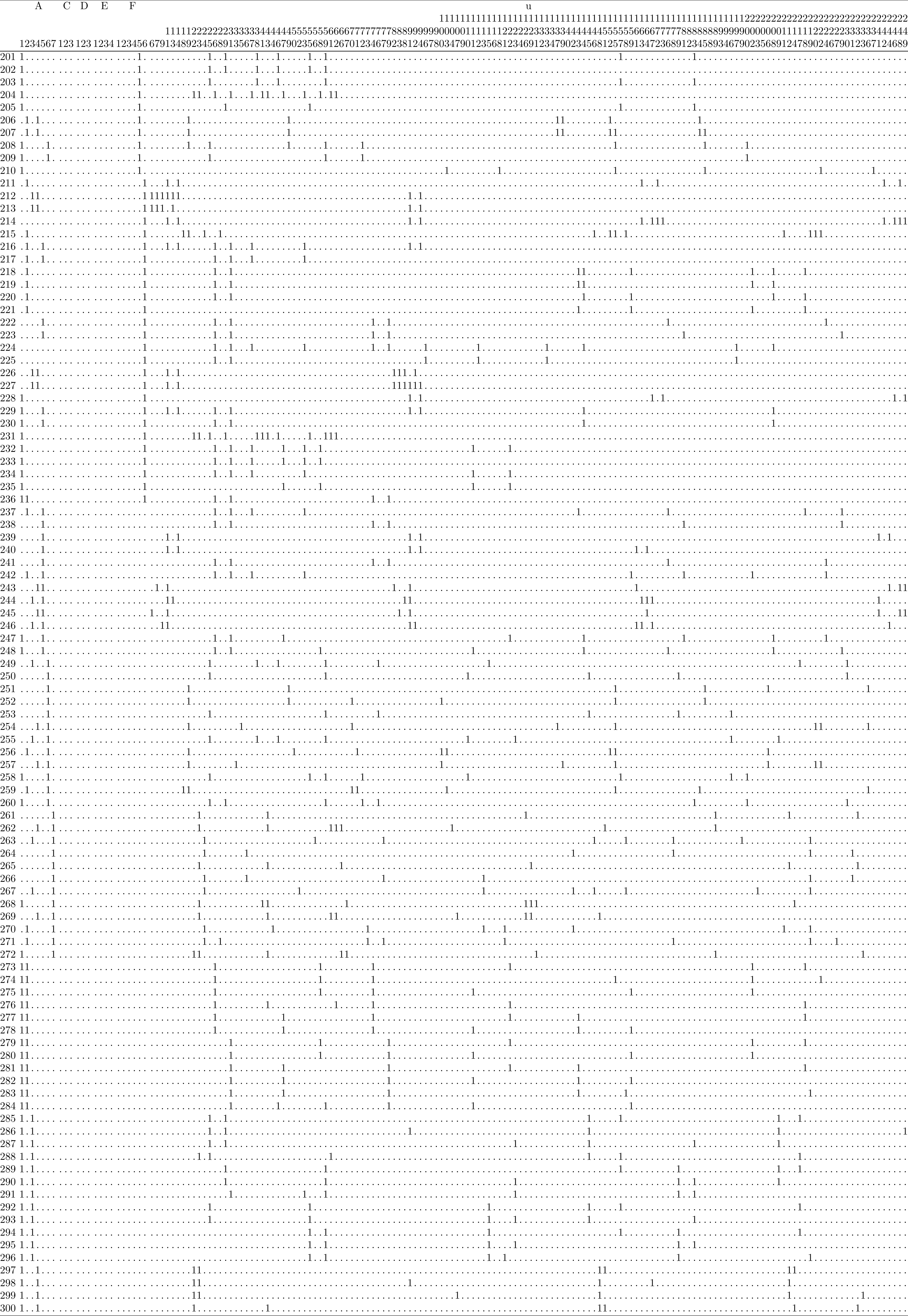}%
\caption{The context $\K''$, objects 201--300 (for attributes, see before Theorem~\ref{thm:context-for-all-centralising-monoids}).}
\label{table:obj-reduced-clarified-context-3}
\end{table}
\begin{table}
\includegraphics[width=\linewidth]{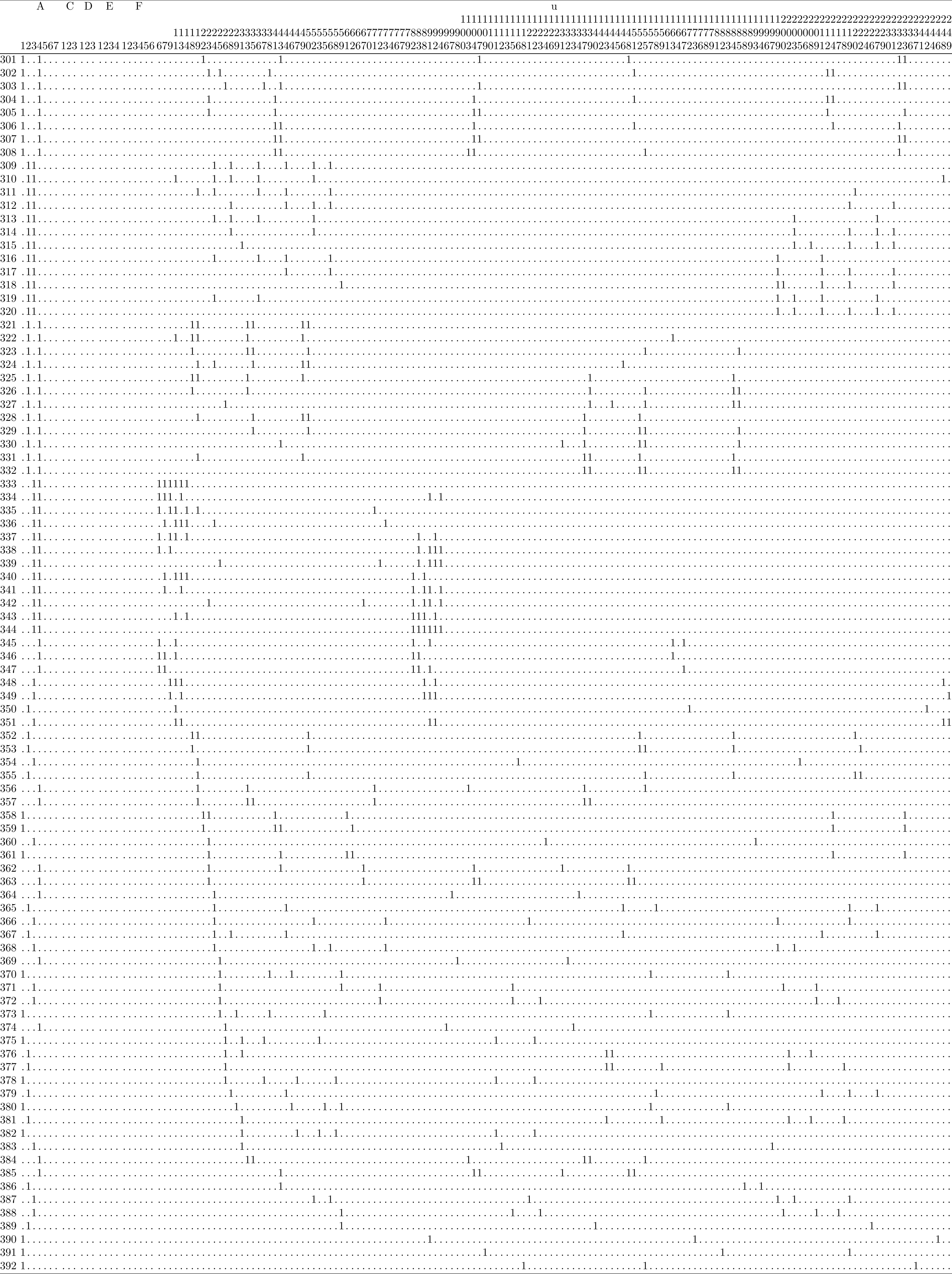}%
\caption{The context $\K''$, objects 301--392 (for attributes, see before Theorem~\ref{thm:context-for-all-centralising-monoids}).}
\label{table:obj-reduced-clarified-context-4}
\end{table}
\begin{table}
\includegraphics[width=\linewidth]{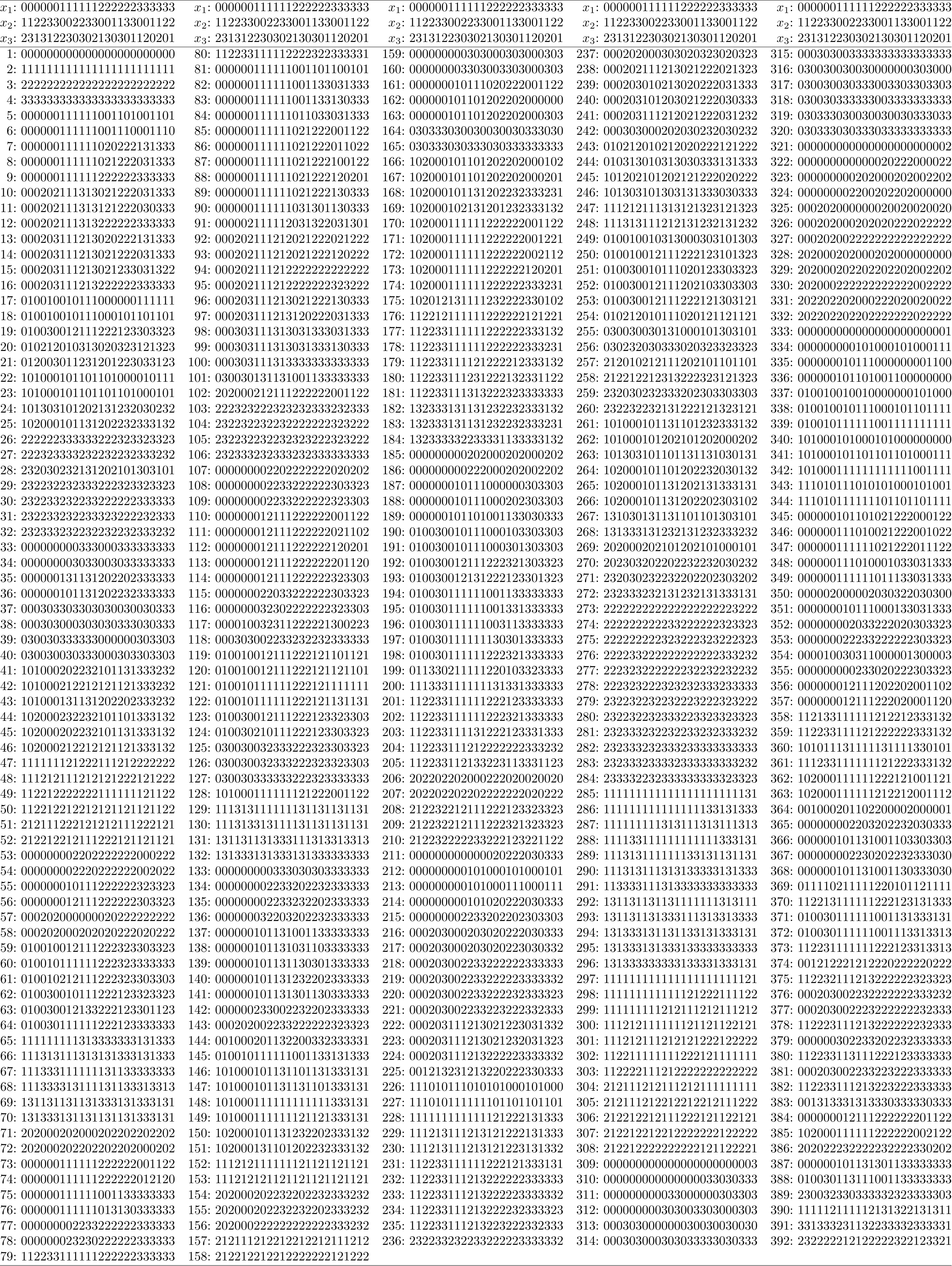}%
\caption{Objects of the context $\K''$, given as value tables of
tuples~$(x_1,x_2,x_3)\in\sigma$.}
\label{tbl-obj-K2}
\end{table}
\begin{table}
\includegraphics[width=\linewidth]{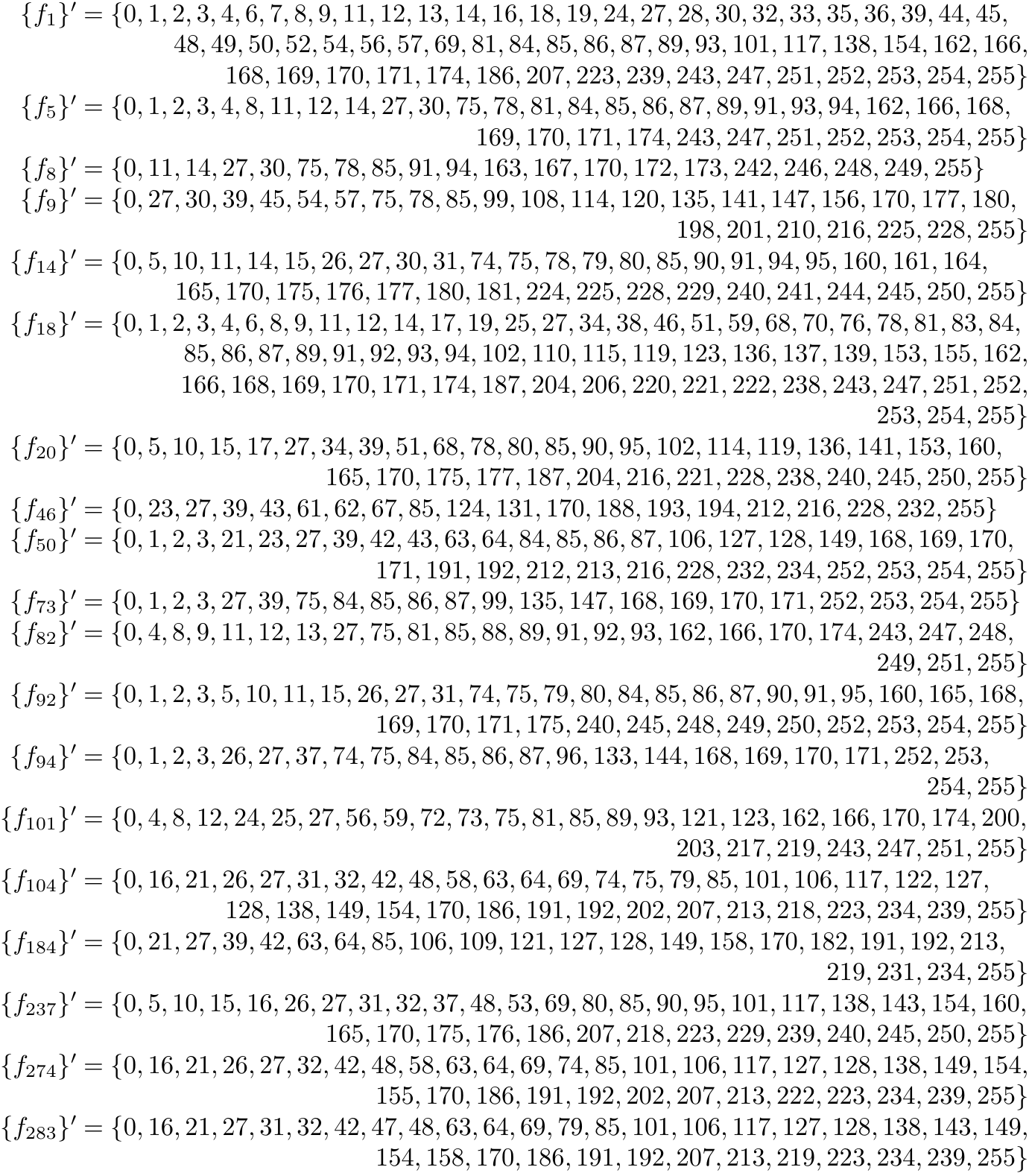}%
\caption{Maximal centralising monoids
$\{f_i\}'\defeq\Fn[1]{\cent{\{f_i\}}{}}$ with majority
witnesses~$f_i$, represented up to conjugacy by inner automorphisms
(for the index~$i$ refer to Table~\ref{tbl-obj-K2}, the operations
$s=u_n\in\{f_i\}'$ are coded by integers \m{n=\sum_{j=0}^3 s(j)4^{3-j}})}
\label{tbl:maximal-monoids-representatives}
\end{table}
\begin{table}
\includegraphics[width=0.895\linewidth]{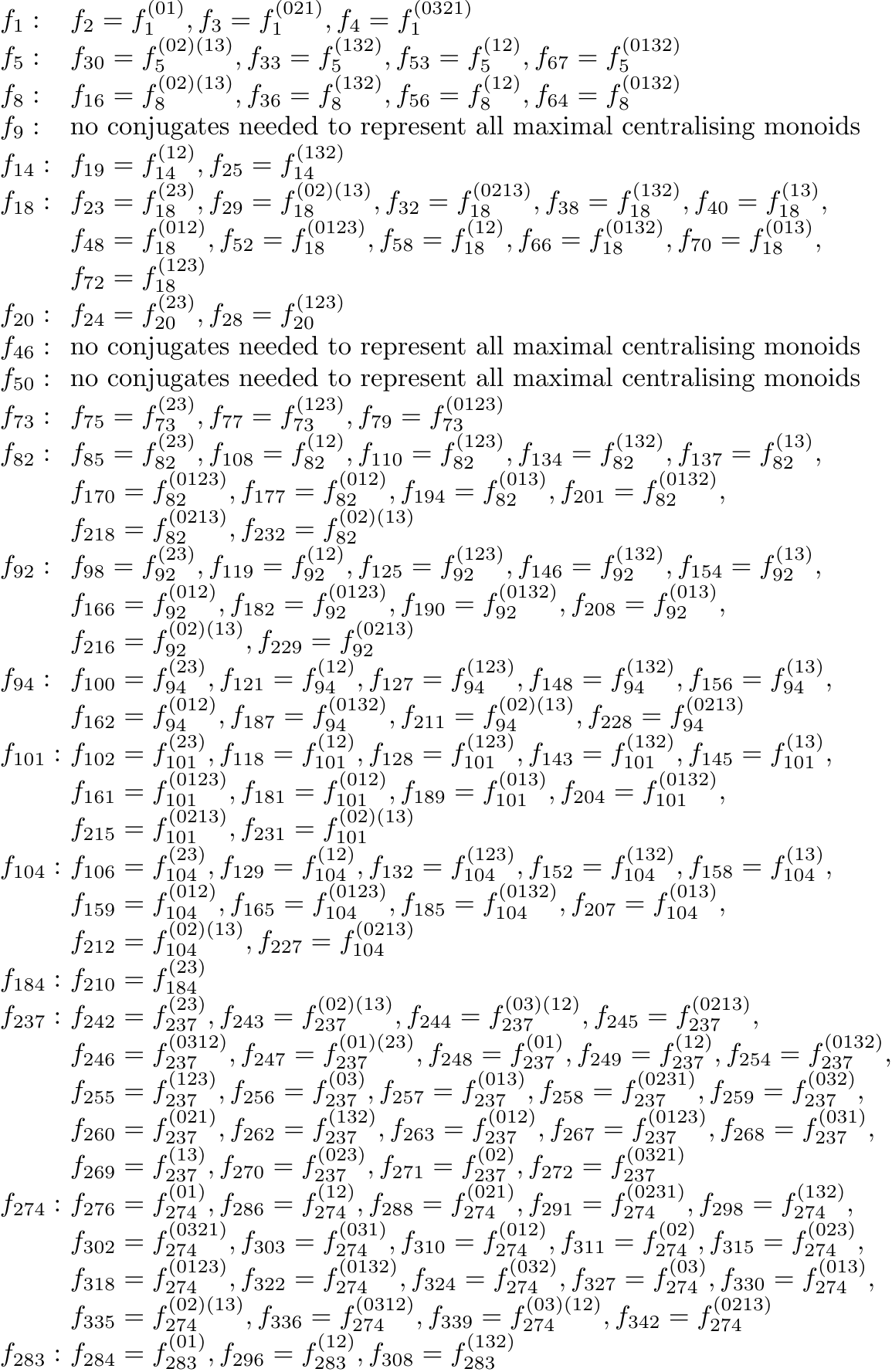}%
\caption{Majority witnesses~$f_i$ for maximal centralising monoids,
represented as conjugates~$f_j^s$ (with $s\in\Sym(4)$) of the witnesses
from Table~\ref{tbl:maximal-monoids-representatives}
(for the indices~$i,j$ refer to Table~\ref{tbl-obj-K2})}
\label{tbl:maximal-monoids-conjugates}
\end{table}
\begin{table}
\includegraphics[width=0.95\linewidth]{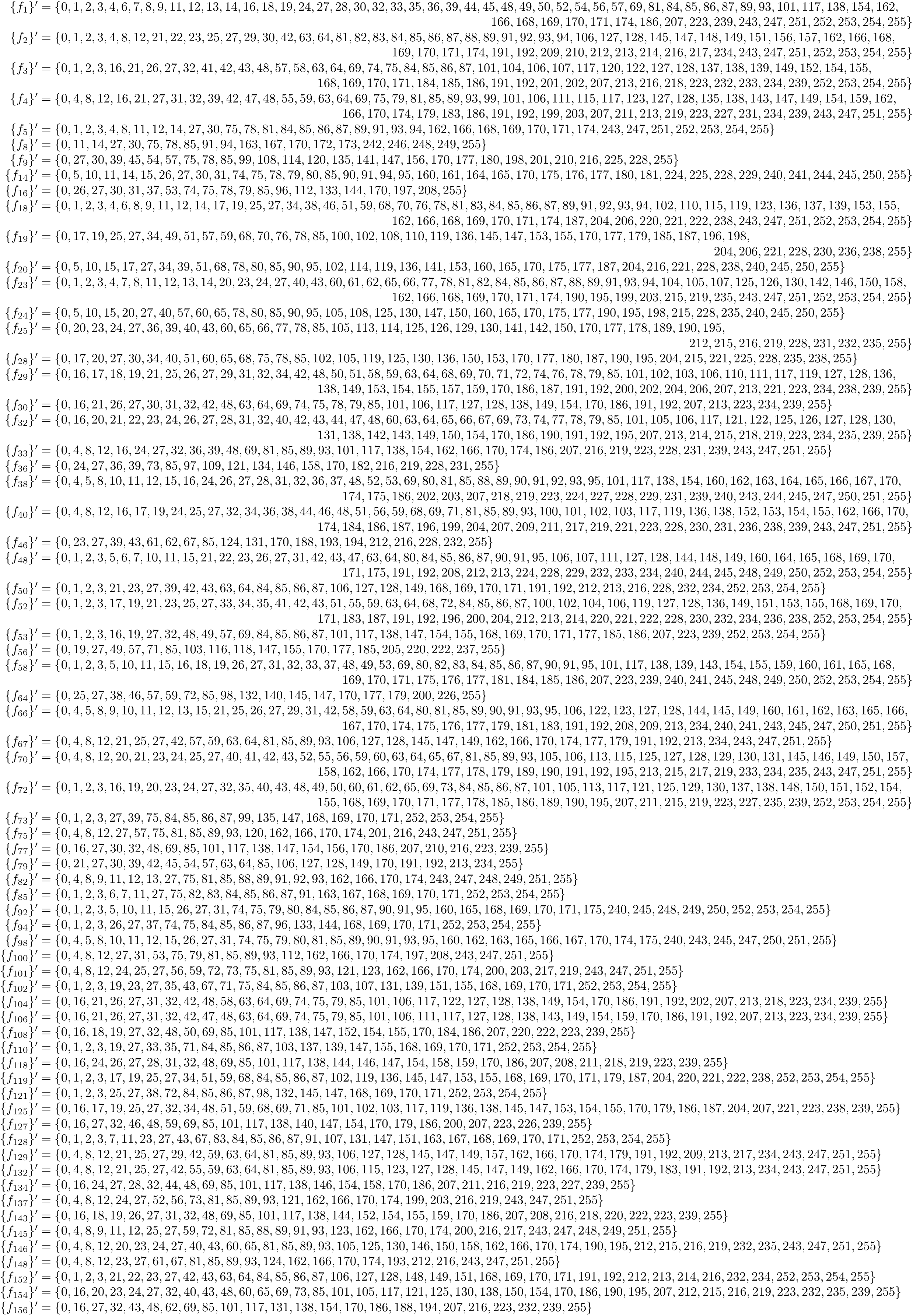}%
\caption{Maximal centralising monoids
$\{f_i\}'\defeq\Fn[1]{\cent{\{f_i\}}{}}$ with majority
witnesses~$f_i$, first part
(for the index~$i$ refer to Table~\ref{tbl-obj-K2}, the operations
$s=u_n\in\{f_i\}'$ are coded by integers \m{n=\sum_{j=0}^3 s(j)4^{3-j}})}
\label{tbl:maximal-monoids-1}
\end{table}
\begin{table}
\includegraphics[width=\linewidth]{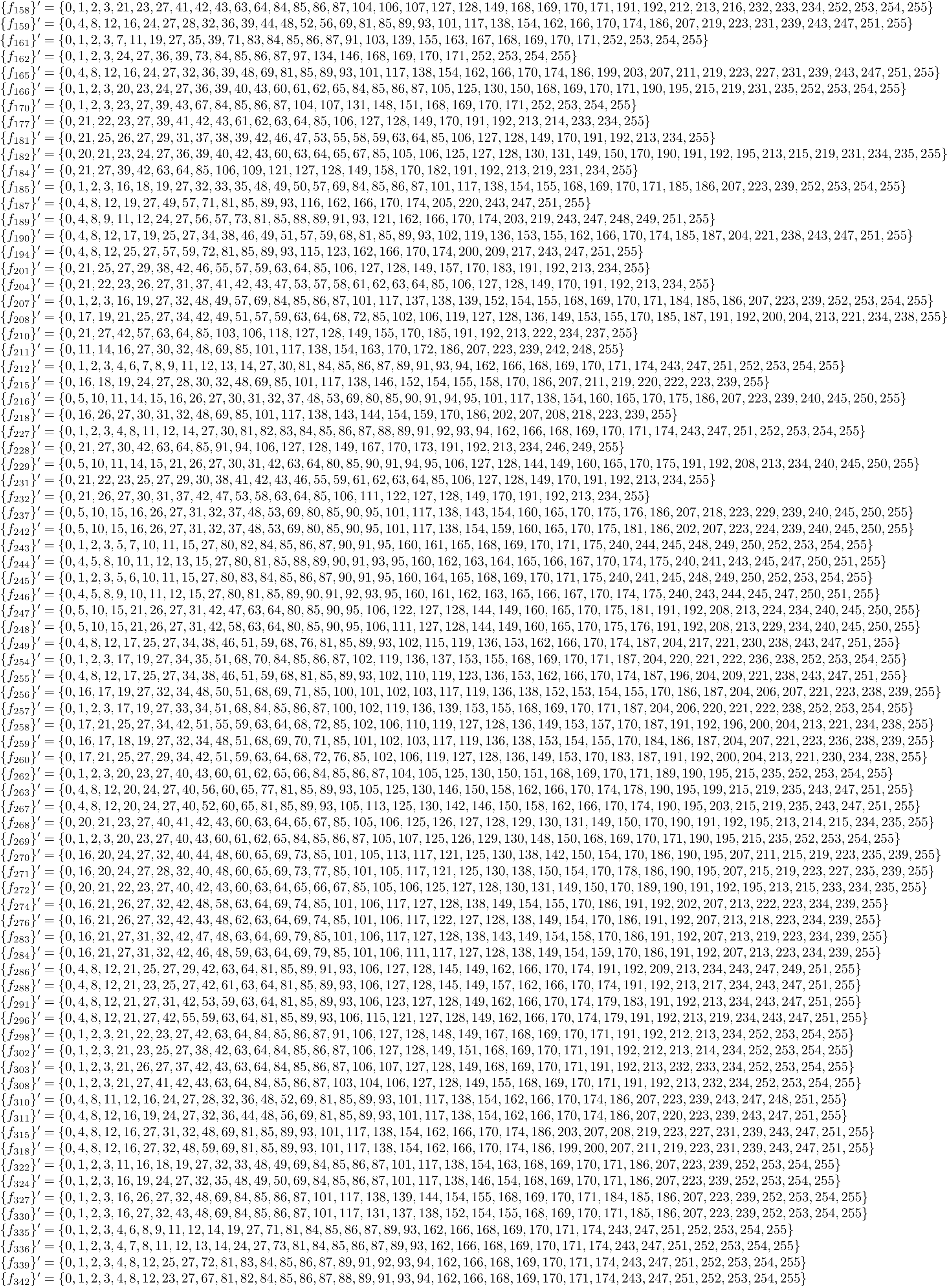}%
\caption{Maximal centralising monoids
$\{f_i\}'\defeq\Fn[1]{\cent{\{f_i\}}{}}$ with majority
witnesses~$f_i$, second part
(for the index~$i$ refer to Table~\ref{tbl-obj-K2}, the operations
$s=u_n\in\{f_i\}'$ are coded by integers \m{n=\sum_{j=0}^3 s(j)4^{3-j}})}
\label{tbl:maximal-monoids-2}
\end{table}

\end{document}

%% file: centralisingmonoids4maj_ref.tex
\def\ocirc#1{\ifmmode\setbox0=\hbox{$#1$}\dimen0=\ht0 \advance\dimen0
  by1pt\rlap{\hbox to\wd0{\hss\raise\dimen0
  \hbox{\hskip.2em$\scriptscriptstyle\circ$}\hss}}#1\else {\accent"17 #1}\fi}
  \def\Palatalization#1{\bgroup\fontencoding{T1}\selectfont\v{#1}\egroup}
  \def\Dj{\bgroup\fontencoding{T1}\selectfont\DJ\egroup}
  \providecommand*{\doi}[1]{\href{http://dx.doi.org/\detokenize{#1}}{\detokenize{#1}}}
  \ifx\SetBibliographyCyrillicFontfamily\undefined\def\SetBibliographyCyrillicFontfamily{\relax}\fi
  \def\rus#1{\foreignlanguage{russian}{\bgroup\fontfamily{cmr}\fontencoding{T2A}\selectfont#1\egroup}}
  \ifx\AvailableAt\undefined\def\AvailableAt{available from}\fi
  \ifx\AVAILABLEat\undefined\def\AVAILABLEat{Available from}\fi
  \ifx\OnlineAvblAt\undefined\def\OnlineAvblAt{online}\fi
  \ifx\PreprintAvblAt\undefined\def\PreprintAvblAt{pre-print}\fi
\providecommand{\bysame}{\leavevmode\hbox to3em{\hrulefill}\thinspace}
\providecommand{\MR}{\relax\ifhmode\unskip\space\fi MR }
\providecommand{\MRhref}[2]{%
  \href{http://www.ams.org/mathscinet-getitem?mr=#1}{#2}
}
\providecommand{\href}[2]{#2}